\def\draft{n}
\documentclass{amsart}
\usepackage[headings]{fullpage}
\usepackage{amssymb,epic,eepic,epsfig}%,amsbsy,amsmath,bbold}
\usepackage{graphicx}
\usepackage{texdraw}

\theoremstyle{plain}

\newtheorem{theorem}{Theorem}

\newtheorem{proposition}{Proposition}[section]
\newtheorem{lemma}[proposition]{Lemma}

\newtheorem{conjecture}{Conjecture}

\theoremstyle{definition}
\newtheorem{definition}[proposition]{Definition}

\newtheorem{question}{Question}

\theoremstyle{remark}

\newtheorem{remark}[proposition]{Remark}

\def\printname#1{
        \if\draft y
                \smash{\makebox[0pt]{\hspace{-0.5in}
                        \raisebox{8pt}{\tt\tiny #1}}}
        \fi
}

\newlength{\standardunitlength}
\catcode`\@=11
\long\def\@makecaption#1#2{%
     \vskip 10pt

\setbox\@tempboxa\hbox{%\ifvoid\tinybox\else\box\tinybox\fi
       \small\sf{\bfcaptionfont #1. }\ignorespaces #2}%
     \ifdim \wd\@tempboxa >\captionwidth {%
         \rightskip=\@captionmargin\leftskip=\@captionmargin
         \unhbox\@tempboxa\par}%
       \else
         \hbox to\hsize{\hfil\box\@tempboxa\hfil}%
     \fi}
\font\bfcaptionfont=cmssbx10 scaled \magstephalf
\newdimen\@captionmargin\@captionmargin=2\parindent
\newdimen\captionwidth\captionwidth=\hsize
\catcode`\@=12

\def\lbl#1{\label{#1}\printname{#1}}

\def\BN{\mathbb N}
\def\BZ{\mathbb Z}

\def\BQ{\mathbb Q}
\def\BR{\mathbb R}
\def\BC{\mathbb C}

\def\calT{\mathcal T}

\def\Ga{\Gamma}

\def\ga{\gamma}

\def\la{\langle}
\def\ra{\rangle}

\def\Ga{\Gamma}

\def\d{\delta}

\def\longto{\longrightarrow}

\def\SL{\mathrm{SL}}

\def\calA{\mathcal{A}}

\def\ep{\epsilon}
\def\deg{\mathrm{deg}}
\def\js{\mathrm{js}}
\def\bs{\mathrm{bs}}

\def\U{\mathrm{U}}

\def\loc{\mathrm{loc}}

\begin{document}

%%%%%%%%%%%%%%%%%%%%%%{page1}

\title[Knots %$q$-holonomic sequences 
and tropical curves]{Knots %$q$-holonomic sequences 
and tropical curves}
%Remarks on the AJ Conjecture, the Slope Conjecture and 
%Tropicalization]{Remarks on the AJ Conjecture, the Slope Conjecture and 
%Tropical Geometry}
\author{Stavros Garoufalidis}
\address{School of Mathematics \\
         Georgia Institute of Technology \\
         Atlanta, GA 30332-0160, USA \\ 
         {\tt http://www.math.gatech} \newline {\tt .edu/$\sim$stavros } }

\thanks{The author was supported in part by NSF. \\
\newline
1991 {\em Mathematics Classification.} Primary 57N10. Secondary 57M25.
}

\date{June 14, 2010} %\today }

%\dedicatory{\large{\bf Private notes. Please do not 
%distribute under any circumstances!}}

\begin{abstract}
A sequence of rational functions in a variable $q$ is $q$-holonomic if it
satisfies a linear recursion with coefficients polynomials in $q$ and $q^n$.
In the paper, we assign a tropical curve to every $q$-holonomic sequence,
which is closely related to the degree of the sequence with respect to $q$.
In particular, we assign a tropical curve to every knot which is 
determined by the Jones polynomial of the knot and its parallels.
The topical curve explains the relation between the AJ Conjecture and
the Slope Conjecture (which relate the Jones polynomial of a knot
and its parallels to the $\SL(2,\BC)$ character variety and to slopes
of incompressible surfaces). Our discussion 
predicts that the tropical curve is dual to a Newton subdivision of the 
$A$-polynomial of the knot. We compute explicitly the tropical curve for 
the $4_1$, $5_2$ and $6_1$ knots and verify the above prediction.
\end{abstract}

\maketitle

\tableofcontents

\section{Introduction}
\lbl{sec.intro}

\subsection{What is a $q$-holonomic sequence?}
\lbl{sub.qholo}

A sequence of rational functions $f_n(q) \in \BQ(q)$
in a variable $q$ is $q$-{\em holonomic} if it
satisfies a linear recursion with coefficients polynomials in $q$ and $q^n$.
In other words, we have

\begin{equation}
\lbl{eq.qholo}
\sum_{i=0}^d a_i(q^n,q) f_{n+i}(q)=0
\end{equation}
where the coefficients $a_i(M,q) \in \BZ[M,q]$ are polynomials for 
$i=0,\dots,d$ where $a_d(M,q) \neq 0$. The term was coined by
Zeilberger in \cite{Z} and further studied in \cite{WZ}.
$q$-holonomic sequences appear in abundance in Enumerative Combinatorics;
\cite{PWZ,St}. The fundamental theorem of Wilf-Zeilberger states that 
a multi-dimensional finite sum of a (proper) $q$-hyper-geometric term
is always $q$-holonomic; see \cite{WZ,Z,PWZ}. Given this result,
one can easily construct $q$-holonomic sequences. 
Combining this fundamental theorem 
with the fact that many {\em state-sum} invariants in Quantum 
Topology are multi-dimensional sums of the above shape, it follows
that Quantum Topology provides us with a plethora of $q$-holonomic
sequences of {\em natural origin}; \cite{GL}. For example, the sequence
of {\em Jones polynomials} of a knot and its parallels 
which we will study below (technically,
the colored Jones function) is $q$-holonomic. 

The goal of our paper is to assign a tropical curve to a $q$-holonomic
sequence.  To motivate the connection between $q$-holonomic sequences
and tropical curves, we will write Equation 
\eqref{eq.qholo} in operator form using
the operators $M,L$ which act on a sequence
$f_n(q) \in \BQ(q)$ by
$$
(M f)_n(q)=q^n f_n(q), \qquad (L f)_n(q)=f_{n+1}(q).
$$ 
It is easy to see that $LM=q ML$ generate the $q$-{\em Weyl algebra} 
\begin{equation}
\lbl{eq.weyl}
W=\BZ[q^{\pm 1}]\la M,L \ra/(LM-qML)
\end{equation}
Equation \eqref{eq.qholo} becomes
\begin{equation}
\lbl{eq.qholo22}
P f =0 
\end{equation}
where
\begin{equation}
\lbl{eq.qholo2}
P =\sum_{i=0}^d a_i(M,q) L^i \in W.
\end{equation}
In other words, Equation \eqref{eq.qholo2} says that $P$ annihilates $f$.
Although a $q$-holonomic sequence $f$ is annihilated by many operators
$P \in W$, it was observed in \cite{Ga2} that
it is possible to canonically choose an operator $P_f$ with coefficients
$a_i(M,q) \in \BZ[M,q]$. Likewise, there is a 
unique non-homogeneous linear recursion
relation of the form $P_f f=b_f$ where $b_f \in \BZ[M,q]$.
For a detailed definition, see Section \ref{sec.weyl} below.

\begin{definition}
\lbl{def.nhom}
We call $P_f$ and $(P^{nh}_f,b_f)$ 
the {\em homogeneous} and the {\em non-homogeneous}
annihilator of the $q$-holonomic sequence $f$.
\end{definition}

\subsection{What is a tropical curve?}
\lbl{sub.curve}

In this section we will recall the definition of a tropical curve.
For a survey on tropical curves,
see \cite{RGST,SS}. With those conventions, a {\em tropical polynomial}
$P: \BR^2 \longto \BR$ is a function of the form:
\begin{equation}
\lbl{eq.Pxy}
P(x,y)=\min\{a_1 x + b_1 y + c_1, \dots, a_r x + b_r y + c_r \}
\end{equation}
where $a_i,b_,c_i$ are rational numbers for $i=1,\dots,r$.
$P$ is convex and piecewise linear. The {\em tropical curve} $\calT(P)$ 
of the tropical polynomial $P$ is
the set of points $(x,y) \in \BR^2$ such that $P$ is not linear at $(x,y)$.
Equivalently, $\calT(P)$ is the set of points where the minimum is
attained at two or more linear functions. 
A {\em rational graph} $\Ga$ is a finite union of
rays and segments whose endpoints and directions are rational numbers, 
and each ray has a positive integer multiplicity.
A {\em balanced rational graph} is defined in \cite[Eqn.10]{RGST}: at every
vertex the sum of the slope vectors with multiplicities adds to zero.
Every tropical curve is a balanced rational graph and vice-versa; see
\cite[Thm.3.6]{RGST}. Tropical curves are very computable objects.
For example, the vertices of a rational curve are the points $(x,y)$ where
the minimum in \eqref{eq.Pxy} is attained at least three times. The
coordinates of such points can be solved by solving a system of linear 
equations. An explicit algorithm to
compute the vertices and the slopes of a tropical curve is given in
\cite[Sec.3]{RGST}, and a computer implementation in {\tt Singular}
is available from \cite{Ma}. This allows us to compute the tropical curves
of the $4_1$, $5_2$ and $6_1$ knots in Sections \ref{sub.41}-\ref{sub.61non} 
below. In the case of the $6_1$ knot,
the non-homogeneous tropical curve is defined by an explicit polynomial
with $r=346$ terms.

Tropical curves arise from 2-variable polynomials $P_t(x,y)$ whose 
coefficients depend on an additional parameter $t$ as follows. Consider
\begin{equation}
\lbl{eq.Pxyt}
P_t(x,y)=\sum_{i=1}^r \ga_i(t) x^{a_i} y^{b_i}
\end{equation}
where $\ga_i(t)$ are algebraic functions of $t$ with order at $t=0$
equal to $c_i$. Then, the corresponding tropical polynomial is given by
\eqref{eq.Pxy}. $P_t(x,y)$ gives rise to two Newton polytopes:
\begin{itemize}
\item
The 3-dimensional Newton polytope $N_P$, i.e., the convex hull
of the exponents of $(x,y,t)$ in $P_t(x,y)$.
\item
The 2-dimensional Newton polygon $N_{P,0}$, i.e.,
the convex hull of the exponents of $(x,y)$ in $P_t(x,y)$.
\end{itemize}
In fact, $N_{P,0}$ is the image of $N_P$ under the projection map
$(x,y,t)\longto (x,y)$. The {\em lower faces} of $N_P$ give rise to a 
Newton subdivision of $N_{P,0}$ which is combinatorially dual to
the tropical curve $\calT(P)$; see \cite{RGST}.

The polynomials $P_t(x,y)$ appear frequently in numerical problems
of {\em Path Homotopy Continuation} where one is interested to
connect $P_0(x,y)$ to $P_1(x,y)$. They also appear in {\em Quantization
problems} in Physics, where $t$ (or $\log t$) plays the role of Planck's 
constant. We will explain below 
that they also appear in Quantum Topology, and they
are a natural companion of the AJ and the Slope Conjecture.

\subsection{The tropical curve of a $q$-holonomic sequence}
\lbl{sub.tropholo}

In this section we associate a tropical surve to a $q$-holonomic sequence.
The main observation is that an element of the $q$-Weyl algebra is
a polynomial in 3 variables $M,L,q$. Two of those $q$-commute (i.e.,
satisfy $LM=qML$) but we can always sort the powers of $L$ to the right
and the powers of $M$ to the left. In other words, there is an
{\em additive} map
\begin{equation}
\lbl{eq.additive}
\BZ[q^{\pm 1}]\la M,L \ra/(LM-qML) \longto \BZ[M,L,q^{\pm 1}]
\end{equation}
Let us change variables $(x,y,1/t)=(L,M,q)$ and ignore the coefficients
of the monomials of $x^i y^j t^k$, and record only their exponents.
They give rise to a tropical curve. Explicitly, 
let
\begin{equation}
\lbl{eq.Pijk}
P=\sum_{(i,j,k) \in \calA} a_{i,j,k} \, q^k M^j L^i \in W
\end{equation}
denote an element of the $q$-Weyl elgebra, where $\calA$ is a finite set
and $a_{i,j,k} \in \BZ\setminus\{0\}$ for all $(i,j,k) \in \calA$. 

\begin{definition}
\lbl{def.ptp}
There is a map
\begin{equation}
\lbl{eq.ptp}
W \longto \{\text{Tropical Curves in $\BR^2$} \},
\qquad P \mapsto \Ga_P
\end{equation}
which assigns to $P$ in \eqref{eq.ptp} the 
tropical polynomial $P_t(x,y)$ given by:
$$
P_t(x,y)=\min_{(i,j,k) \in \calA} \{ i x + j y -k \}
$$
$\Ga_P$ is the tropical curve of $P_t(x,y)$.
\end{definition}

Combining Definitions \ref{def.nhom} and \ref{def.ptp}
allows us to assign a tropical curve to a $q$-holonomic sequence $f$.

\begin{definition}
\lbl{def.tropholo}
\rm{(a)} If $f$ is a $q$-holonomic sequence, let
$\Ga_f$ and $\Ga^{nh}_f$ denote the tropical curves of $P_f(y,x,1/t)$
and $P^{nh}_f(y,x,1/t)$ respectively, where $P_f(M,L,q)$ and $P^{nh}_f(M,L,q)$
are given in Definition \ref{def.nhom}.
\end{definition}

The tropical curve $\Ga_f$ of a $q$-holonomic sequence $f$ is closely
related to the degree (with respect to $q$) of the sequence of rational
functions $f_n(q)$. If $\d_n=\deg_q(f_n(q))$ denotes this degree, then
it was shown in \cite{Ga4} that for large enough $n$, $\d_n$ is a 
quadratic quasi-polynomial with slope recorded by the rays of the 
tropical curve $\Ga_f$.

\subsection{3 polytopes of a $q$-holonomic sequence}
\lbl{sub.3polytopes}

In this section we assign 3 polytopes to a $q$-holonomic sequence.

\begin{definition}
\lbl{def.3poly}
\rm{(a)}
If $P \in W$ is given by Equation \eqref{eq.Pijk}, it defines 3 polytopes:
\begin{itemize}
\item 
$N_P$ is the convex hull 
of the exponents of the polynomial $P(M,L,q)$ with respect
to the variables $(M,L,q)$. 
\item
$N_{P,0}$ is the projection of $N_P$ under the
projection map $(M,L,q) \longto (L,M)$. 
\item
$N_{P,1}$ is the 
convex hull of the exponents of the polynomial $P(L,M,1)$. 
\end{itemize}
\rm{(b)} If $f$ is a $q$-holonomic sequence, its annihilator $P_f$
gives rise to the polytopes $N_{P_f}$, $N_{P_f,0}$ and $N_{P_f,1}$.
\end{definition}
Note that $N_P$ is a 3-dimensional convex lattice polytope, and 
$N_{P,0}, N_{P,1}$ are 2-dimensional convex lattice polygons. 
Since every exponent of $P(M,L,1)$
comes from some exponents of $P(M,L,q)$, it follows that
\begin{equation}
\lbl{eq.N01}
N_{P,1} \subset N_{P,0}
\end{equation}

\begin{remark}
\lbl{rem.dual}
It follows by \cite{RGST} that the tropical curve $\Ga_P$ is dual
to a Newton subdivision of $N_{P,0}$.
\end{remark}

We will say that $P(M,L,q)$ is {\em good} if $N_{P,1}=N_{P,0}$. It is easy
to see that goodness is a generic property.

\subsection{The slopes of a $q$-holonomic sequence}
\lbl{sub.slopeqholo}

In this section we discuss the slopes of a $q$-holonomic sequence and
their relation with its tropical curve. The proof of the following theorem
uses differential Galois theory and the Lech-Mahler-Skolem theorem 
from number theory.

\begin{theorem}
\lbl{thm.0}\cite{Ga4}
The degree with respect to $q$ of a $q$-holonomic sequence 
$f_n(q) \in \BQ(q)$ 
is given (for large values of $n$) by a quadratic quasi-polynomial.
\end{theorem}
Recall that a {\em quadratic quasi-polynomial}
is a function of the form:
\begin{equation}
\lbl{eq.pp}
p: \BN \longto \BN, \qquad p(n)=\ga_2(n) \binom{n}{2} 
+ \ga_1(n) n + \ga_0(n)
\end{equation}
where $\ga_j(n)$ are rational-valued periodic functions of $n$. 
Quasi-polynomials appear in lattice point counting problems, and also
in Enumerative Combinatorics; see \cite{BP,BR,Eh,St} and references therein.

The set of {\em slopes} $s(p)$ of a quadratic quasi-polynomial is the finite 
set of values of the periodic function $\ga_2(n)$. These are essentially the
quadratic growth rates of the quasi-polynomial. More precisely,
recall that $x \in \BR$ is a {\em cluster point} of a sequence $(x_n)$ of 
real numbers if for every $\ep>0$
there are infinitely many indices $n \in \BN$ such that $|x-x_n| < \ep$. 
Let $\{x_n\}'$ denote the set of {\em cluster points} of a sequence $(x_n)$.
It is easy to show that for every quadratic quasi-polynomial $p$ we have:

\begin{equation}
\lbl{eq.sp}
s(p)=\{ \frac{2}{n^2}p(n) \,\, | n \in \BN \}' \subset \BQ
\end{equation}

Given a $q$-holonomic sequence $f_n(q) \in \BQ(q)$, let $s(f)$ denote the 
slopes of the quadratic quasi-polynomial $\deg_q f_n(q)$.
Let $s(N)$ 
denote the set of slopes of the edges of a convex polygon $N$ in the plane.
%Likewise, given a polynomial $p(x,y) \in \BQ[x,y]$, let $s(p)$ denote
%the set of slopes of the Newton polygon of $p(x,y)$. 
The next proposition relates the slopes of a $q$-holonomic sequence
with its tropical curve. See also \cite[Prop.4.4]{Ga4}.

\begin{proposition}
\lbl{prop.1}
If $f$ is $q$-holonomic, then $s(f) \subset -s(N_{P_f,0})$.
\end{proposition}

\begin{proof}
Let $\d(n)=\deg_q f_n(q)$ denote the degree of
$f_n(q)$ with respect to $q$, and let $P$ denote the annihilator of $f$.
We expand $P$ in terms of monomials as in Equation \eqref{eq.Pijk}.
For every monomial $q^k M^j L^i$ and every $n$ we have
$$
\deg_q((q^k M^j L^i) f_n(q))=k+jn+\d(n+i).
$$
Since $P$ annihilates $f$, for every $n$ the following
maximum is attained {\em at least twice} (from now on, twice will mean at 
least twice as is common in Tropical Geometry):

\begin{equation}
\lbl{eq.maxP}
\max_{(i,j,k)}\{ jn+k+\d(n+i)\}
\end{equation}
Subtracting $\d(n)$, it follows that the maximum is obtained twice:
\begin{equation}
\lbl{eq.maxP2}
\max_{(i,j,k)} \{ jn+k+\d(n+i)-\d(n)\}
\end{equation}
Now $\d(n)$ is a quadratic quasi-polynomial given by
$$
\d(n)=\ga_2(n) \binom{n}{2} 
+ \ga_1(n) n + \ga_0(n)
$$
Theorem \ref{thm.0} implies that for large enough $n$, in a fixed
arithmetic progression, we have $\ga_i(n)=\widehat{\ga}_i$ for $i=1,2$,
thus
$$
\d(n+i)-\d(n)=\widehat{\ga}_2 \, i \, n + \widehat{\ga}_2 \, \binom{i}{2} 
+ \widehat{\ga}_1 \, i
$$
Substituting into \eqref{eq.maxP2}, it follows that for large enough $n$
in an arithmetic progression, the max is obtained twice:
\begin{equation}
\lbl{eq.maxP3}
\max_{(i,j,k)}\{ jn+k+
\widehat{\ga}_2 \, i \, n + \widehat{\ga}_2 \, \binom{i}{2} 
+ \widehat{\ga}_1 \, i \}
\end{equation}
It follows that there exists $(i',j') \neq (i,j)$ such that
\begin{equation}
\lbl{eq.ga2}
\widehat{\ga}_2=-\frac{j-j'}{i-i'}.
\end{equation}
This proves Proposition \ref{prop.1}.
\end{proof}

\section{The $q$-Weyl alegbra and its localization}
\lbl{sec.weyl}

In this section we will discuss some algebraic properties of the $q$-Weyl
algebra and its localization, which will justify Definition \ref{def.nhom}.

Recall the $q$-Weyl algebra from \eqref{eq.weyl}. We will say that an element
$P$ of $W$ is {\em reduced} if it has the form \eqref{eq.qholo2}
where $a_i(M,q) \in \BZ[M,q]$ for all $i$, 
and the greatest common divisor of $a_i(M,q) \in \BZ[M,q]$ is $1$. 

Consider the {\em localized $q$-Weyl algebra} $W_{\loc}$ given by
\begin{equation}
\lbl{eq.wloc}
W_{\loc}=\BQ(M,q)\la L \ra/(Lf(M,q)-f(Mq,q)L)
\end{equation}
It was observed in \cite{Ga2} that $W$ is not a principal left-ideal domain, 
but becomes so after localization; see \cite{Cou}.
If $f$ is a sequence of rational functions, consider the left ideal
$M_f$ 
$$
M_f=\{P \in W_{\loc} \, | Pf =0 \}
$$
$M_f$ is a principal ideal, which is nonzero if $f$ is $q$-holonomic. 
Let $P'$ denote the monic generator of $M_f$.
Left multiply it by a polynomial in $M,q$ so as to obtain a reduced
annihilator $P_f$ of $f$. 

Now, we discuss non-homogeneous recursion relations of the form
$$
\sum_{i=0}^d a_i(q^n,q) f_{n+i}(q)=b(q^n,q)
$$
where $a_i(M,q), b(M,q) \in \BQ(M,q)$ for all $i$. In operator form,
we can write the above recursion as
$$
Pf=b.
$$
Consider the left ideal
\begin{equation}
\lbl{eq.Mhnf}
M^{nh}_f=\{ P \in W_{\loc} \,\, | \exists b \in \BQ(M,q) \, : P f=b \}
\end{equation}
It is easy to see that $M^{nh}_f$ is a left ideal. If $f$ is $q$-holonomic,
$M^{nh}_f \neq 0$. Let $P''$ denote the monic generator of $M^{nh}_f$. There
exists $b'' \in \BQ(M,q)$ such that
$$
P'' f = b''
$$
There are two cases: $b'' \neq 0$ or $b''= 0$. If $b'' \neq 0$, then
dividing by $b''$ we obtain that $1/b \cdot P'' f=1$. We left multiply 
both sides by a polynomial in $M,q$ so as to obtain $P^{nh}_f f= b_f$
where $P^{nh}$ is reduced. If $b''=0$ then multiply by a polynomial in $M,q$
so as to obtain $P^{nh}_f f=0$ and define $b_f=0$ in tha case.
This concludes Definition \ref{def.nhom}.

The next lemma relates the homogeneous and the non-homogeneous
annihilator of a $q$-holonomic sequence.
It is well-known that one can convert an non-homogeneous recursion relation
$Pf=b$ where $b \neq 0$ into a homogeneous recursion relation of order
one more. Indeed, $Pf=b$ where $b \neq 0$ is equivalent to 
$$
(L-1)b^{-1} P f=0
$$ 
This implies the following conversion between $(P^{nh}_f,b_f)$ and $P_f$.
Fix a $q$-holonomic sequence $f_n(q) \in \BQ(q)$.

\begin{lemma}
\lbl{lem.nh}
\rm{(a)} If $b_f =0$ then $P^{nh}_f=P_f$. If $b_f \neq 0$, then $P^{nh}_f$
is obtained by clearing denominators of $(L-1)b_f^{-1} P_f$ by putting
the powers of $L$ on the right and the elements of $\BQ(M,q)$ on the left.
\newline
\rm{(b)} If $P_f$ is not left divisible by $L-1$ in $W$, then $P_f=P^{nh}_f$
and $b_f=0$. If $P_f$ is left divisible by $L-1$ in $W$, then $P_f=(L-1)
Q_f$ and if $d$ is the common denominator of $Q_f$, then $(d Q_f, d)= 
(P^{nh}_f,b_f)$. 
\end{lemma}

\begin{definition}
\lbl{def.fhom}
We say that a $q$-holonomic sequence $f$ is {\em homogeneous} 
if $b_f=0$--else $f$ is non-homogeneous.
\end{definition}
In other words, a $q$-holonomic sequence $f$ is {\em homogeneous} if
and only if $P_f$ is left-divisible by $L-1$ in $W$.

\section{Quantum Topology}
\lbl{sec.qt}

\subsection{The tropical curve of a knot}
\lbl{sub.cjones}

Quantum Topology is a source of $q$-holonomic sequences attached to
knotted 3-dimensional objects. 
Let $J_{K,n}(q) \in \BZ[q^{\pm 1}]$ 
denote the {\em colored Jones polynomial} of a knot $K$ in 3-space, 
colored by the $(n+1)$-dimensional 
irreducible representation of $\mathfrak{sl}_2$ and normalized to be $1$ at 
the unknot; \cite{Jo,Tu}. The sequence $J_{K,n}(q)$ for $n=0,1,\dots$ 
essentially encodes the Jones polynomial of a knot and all of its parallels;
see \cite{Tu}. In \cite[Thm.1]{GL} it was shown that the sequence 
$J_{K,n}(q)$ of colored Jones polynomials of a knot $K$ is $q$-holonomic.

\begin{definition}
\lbl{def.Aq}
\rm{(a)}
If $K$ is a knot, we denote by $A_K(M,L,q)$ and
$(A^{nh}_K(M,L,q),B_K(M,q))$ the homogeneous and the 
non-homogeneous annihilator of the $q$-holonomic sequence $J_{K,n}(q)$. 
These are the non-commutative and the non-homogeneous non-commutative
$A$-polynomials of the knot.
\newline
\rm{(b)} 
If $K$ is a knot, let $\Ga_K$ and $\Ga^{nh}_K$ denote the {\em tropical curves} 
of $A_K$ and $A^{nh}_K$ respectively.
\end{definition}
The non-homogeneous non-commutative $A$-polynomial of a knot appeared 
first in \cite{GS}.

\subsection{The AJ Conjecture}
\lbl{sub.AJ}

The AJ Conjecture (resp. the Slope Conjecture) relates the Jones polynomial
of a knot and its parallels to the $\SL(2,\BC)$ character variety (resp. 
to slopes of incompressible surfaces) of the knot complement. We will
relate the two conjectures using elementary ideas
from Tropical Geometry. 

The $A$-polynomial of a knot is a polynomial in two commuting variables
$M$ and $L$ that essentially encodes the image of the 
$\SL(2,\BC)$ character variety of $K$, projected in $\BC^* \times \BC^*$ by
the eigenvalues of a meridian and longitude of $K$. It was defined in
\cite{CCGLS}.

\begin{conjecture}
\lbl{conj.AJ}\cite{Ga2}
The AJ Conjecture states that
\begin{equation}
\lbl{eq.AJ}
A_K(M,L,1)=B_K(M) A_K(M^{1/2},L)
\end{equation}
where $A_K(M,L)$ is the $A$-polynomial of $K$ and $B_K(M)
\in \BZ[M]$ is a polynomial that depends on $M$ and of course $K$.
\end{conjecture}
The AJ Conjecture is known for infinitely many 2-bridge knots; see \cite{Le}.

It is natural to ask whether the $q$-holonomic sequence $J_{K,n}(q)$ is
of non-homogeneous type or not. Based on geometric information (the
so-called {\em loop expansion} of the colored Jones polynomial, 
see \cite{Ga1}), as well as experimental evidence for all knots whose 
non-commutative $A$-polynomial is known (these are the torus knots in
\cite{Hi} and the twist knots in \cite{GS}) we propose the following
conjecture.

\begin{conjecture}
\lbl{conj.inhom}
For every knot $K$, $J_{K,n}(q)$ is non-homogeneous.
\end{conjecture}
The above conjecture implies that $B_K(M,q) \in \BZ[M,q]\setminus\{0\}$ is an 
invariant which is independent and {\em invisible} from the classical 
$A$-polynomial of the knot. There is a close connection between the 
$B_K(M,q)$ invariant of a knot and the torsion polynomial of the knot
introduced in \cite{DbG}. We will discuss this in a future publication.

\subsection{The Slope Conjecture}
\lbl{sub.slope}

The Slope Conjecture of \cite{Ga3} relates the degree of the colored Jones 
polynomial of a knot and its parallels to slopes of incompressible surfaces in
the knot complement. To recall the conjecture, 
let $\d_K(n)=\deg_q J_{K,n}(q)$ (resp. $\d^*_K(n)=\deg^*_q J_{K,n}(q)$) 
denote the maximum (resp. minimum) {\em degree} of the polynomial 
$J_{K,n}(q) \in \BZ[q^{\pm 1}]$ (or more generally, of a rational function)
with respect to $q$.

For a knot $K$, define the {\em Jones slopes} $\js_K$ by:
\begin{equation}
\lbl{eq.js}
\js_K=\{ \frac{2}{n^2}\d_K(n) \,\, | n \in \BN \}'
\end{equation}
\rm{(b)} Let $\bs_K \subset \BQ \cup \{1/0\}$ 
denote the set of boundary slopes of incompressible 
surfaces of $K$; \cite{Ha,HO}.

\begin{conjecture}
\lbl{conj.slope}\cite{Ga3}
The {\em Slope Conjecture} states that 
for every knot $K$ we have
\begin{equation}
2 \js_K \subset \bs_K.
\end{equation}
\end{conjecture}
Note that the Slope Conjecture applied to the mirror of $K$ implies that
$2 \js^*_K \subset \bs_K$. The Slope Conjecture is known for alternating knots
and torus knots (see \cite{Ga3}), for adequate knots (which include all
alternating knots; see \cite{FKP}), for $(-2,3,n)$ pretzel knots (see 
\cite{Ga3}), and for 2-fusion knots; see \cite{DnG}. A general method
for verifying the Slope Conjecture is discussed in \cite{Ga5,DnG}.

\subsection{The AJ Conjecture and the Slope Conjecture}
\lbl{sub.AJslope}

In this section we will see how the AJ Conjecture relates to the Slope
Conjecture, expanding a comment of \cite[Sec.2]{Ga3}.
We will specialize Definition \ref{def.3poly} to knot theory
when $P=A_K$ is the non-commutative $A$-polynomial of a knot $K$, and
we will denote by $N_K$,  $N_{K,0}$ and $N_{K,1}$ the three polytopes associated 
to $A_K$. Proposition \ref{prop.1} implies that
\begin{equation}
\lbl{eq.e1}
\js_K \subset -N_{K,0}
\end{equation}
Let $\bs^A_K$ denote the slopes of the $A$-polynomial of $K$. 
The AJ Conjecture implies that up to possibly excluding the slope $1/0$
from $2 N_{K,1}$, we have:
\begin{equation}
\lbl{eq.e2}
2 N_{K,1} = \bs^A_K.
\end{equation}
For a careful proof, see Proposition \ref{prop.2} and Remark 
\ref{rem.shift} below.
Culler and Shalen show that edges of the Newton polygon of the $A$-polynomial 
of $K$ give rise to ideal points of the $\SL(2,\BC)$ character variety of $K$;
see \cite{CS,CGLS,CCGLS}. For every ideal point,
Culler and Shalen construct an incompressible surface
whose slope is a boundary slope of $K$; see \cite{CS,CCGLS}. $\bs^A_K$
is the set of the so-called {\em strongly detected boundary slopes} of $K$, 
and satisfies the inclusion:
\begin{equation}
\lbl{eq.e3}
\bs^A_K \subset \bs_K.
\end{equation}
If $A_K(M,L,q)$ is good, then 
\begin{equation}
\lbl{eq.e4}
N_{K,0} = N_{K,1}.
\end{equation}
If $K^*$ denotes the mirror of $K$, then $J_{K^*,n}(q)=K_{K,n}(q^{-1})$ which
implies that $-N_{K,0}=N_{K^*,0}$. Combining Equations 
\eqref{eq.e1}-\eqref{eq.e4},  it follows that
$$
2 \js_K \subset \bs_{K^*}
$$
which is the Slope Conjecture, up to a harmless mirror image.
This derivation also explains two independent factors of $2$, one in Equation 
\eqref{eq.js} and another one in Equation \eqref{eq.AJ}.

\begin{proposition}
\lbl{prop.2}
If the non-commutative $A$-polynomial of $K$ is good, 
and if the AJ Conjecture holds, then 
$\Ga_K$ is dual to the Newton subdivision of the A-polynomial of $K$
(multiplied by a polynomial in $M$).
\end{proposition}

\begin{proof}
Let $P$ denote the non-commutative $A$-polynomial of a knot $K$.
$\Ga_K$ is dual to $N_{P,0}$. If $P$ is good, then $N_{P,0}=N_{P,1}$. With
the notation of Conjecture \ref{conj.AJ}, the AJ Conjecture implies that 
$$
P(M,L,1)=A_K(M^{1/2},L) B_K(M)
$$
where $B_K(M)$ is a polynomial of $M$, and $A_K$ is the $A$-polynomial
of $A$. The Newton polygon of 
of the product of two polynomials is the Minkowski sum of their Newton
polygons. Moreover, the Newton polygon of $B_K(M)$ is a vertical
line segment in the $(L,M)$-plane. It follows that the Newton polygon
of $A_K(M^{1/2},L) B_K(M)$ is the Newton polygon of the $A$-polynomial of
$K$ and its translation by a vertical segment. On the other hand,
the Newton polygon of $P(M,L,1)$ is $N_{P,1}$. The result follows.
\end{proof}

\begin{remark}
\lbl{rem.shift}
Note that the Newton polygon of $A_K(M^{1/2},L) B_K(M)$ is the Newton
polygon of $A_K(M^{1/2},L)$ and its shift by a vertical line segment.
It follows that the slopes of the Newton polygon of $A_K(M^{1/2},L)B_K(M)$
are the slopes of $A_K(M^{1/2},L)$ plus the slope of a vertical segment
(i.e., $1/0$). For concrete examples, see Section \ref{sec.compute}
where the Newton polygons of the non-homogeneous $A$-polynomials
of $4_1,5_2,6_1,8_1$ is shown and it differs from the Newton polygon
of the $A$-polynomial by a shift by a vertical segment. 
\end{remark}

The only knots with explicitly known non-commutative $A$-polynomials 
(homogeneous
and non-homogeneous) are the handful of twist knots $K_p$ of \cite{GS}
for $p=-8,\dots,11$. An explicit check shows that these non-commutative
$A$-polynomials (both the homogeneous and the non-homogeneous) are good.
For details, see Section \ref{sec.compute}.

\section{Quantization and Tropicalization}
\lbl{sec.qtrop}

Quantization is the process of producing
the non-commutative $A$-polynomial of a knot from the usual $A$-polynomial.
In other words, Quantization starts with $P_1(x,y)$ and produces $P_t(x,y)$
as in Equation \eqref{eq.Pxyt}. On the other hand, Tropical Geometry
expands $P_t(x,y)$ at $t=0$ (or equivalently at $q=\infty$) and produces
a tropical curve. Schematically, we have a diagram:

$$
\left(\begin{array}{c}
A\text{-polynomial} \\ q=1 
\end{array}\right)
\stackrel{\text{Classical limit}}\longleftarrow
\left(\begin{array}{c}
\text{non-commutative} \\ A\text{-polynomial} \\ q
\end{array}\right)
\stackrel{\text{Tropicalization}}\longto
\left(\begin{array}{c}
\text{Tropical curve} \\ q=\infty 
\end{array}\right)
$$
Quantization is a map reverse to the Classical Limit map in the above diagram.
Both sides of the above diagram (i.e., the limits at $q=1$ and $q=\infty$)
are classical {\em dual} invariants of the knot. Indeed, 
the tropical curve ought to be dual to a Newton subdivision of the
$A$-polynomial of $K$. This duality is highly nontrivial, even for the
simple case of the $4_1$ knot, computed  in Section \ref{sub.41} below. 

This conjectured duality may be related to the duality between 
Chern-Simons theory (i.e., colored $\U(N)$ polynomials of a knot)
and Enumerative Geometry (i.e., BPS states) of the corresponding
Calabi-Yau 3-fold. For a discussion of the latter duality, see 
\cite{ADKMV,DGKV,LMV,DV} and references therein.

Physics principles concerning 
Quantization of complex Lagrangians in Chern-Simons theory suggest that
the $A$-polynomial of a knot should determine the non-commutative 
$A$-polynomial. In particular, it should determine the polynomial
invariant $B_K(M,q)$ of Definition \ref{def.nhom}, and it should determine
the tropical curves $\Ga_K$ and $\Ga^{hn}_K$. 
%Our intuition suggests that
%perhaps the $A$-polynomial of a knot {\em together} with 
%(a) its tropical curve $\Ga^{nh}_K$ and 
%(b) the $B_K(M,q)$ polynomial might determine the
%non-commutative $A$-polynomial of $K$.

Aside from duality conjectures, let us concentrate on a concrete 
question. It is
well-known that the $A$-polynomial of a knot is a {\em triangulated curve}
in the sense of {\em algebraic K-theory}. In other words, if $X$ is the 
curve of zeros $A_K(M,L)=0$ of the $A$-polynomial then there exist
nonzero rational functions $z_1,\dots,z_r  \in C(X)^*$ in $X$ such that

\begin{equation}
\lbl{eq.triang}
M \wedge L= 2 \sum_{i=1}^r z_i \wedge (1-z_i) \in \wedge^2_{\BZ}(C(X)^*)
\end{equation}
where $C(X)$ is the field of rational functions of $X$ and 
$M, \, L \in C(X)^*$ are the eigenvalues of the meridian and
the longitude. For a proof of \eqref{eq.triang}
(which uses the symplectic nature of
the so-called Neumann-Zagier matrices), see \cite[Lem.10.1]{Ch}. 
For an excellent discussion of triangulated curves $X$
and for a plethora of examples and computations, see \cite{BRVD}.
Geometrically, a triangulation of $X$ comes from an ideal triangulation of the
knot complement with $r$ ideal tetrahedra with shape parameters 
$z_1,\dots,z_r$ which satisfy some gluing equations. The symplectic nature
of these gluing equations, introduced and studied by Neumann and Zagier in 
\cite{NZ}, implies \eqref{eq.triang}. The triangulation of
$X$ has important arithmetic consequences regarding the {\em volume} 
of the knot complement and its Dehn fillings, 
and it is closely related to the {\em Bloch
group} of the complex numbers. It is important to realize that $X$ has
infinitely many triangulations, and in general it is not possible to 
choose a canonical one. In addition, triangulations tend to work well
with hyperbolic knots. On the contrary, the non-commutative $A$-polynomial
and its corresponding tropical curve exist for every knot in 3-space,
hyperbolic or not. 
Let us end with some questions, which aside from its theoretical interest, 
may play a role in the Quantization of the $A$-polynomial.

\begin{question}
\lbl{que.1}
Is the tropical curve $\Ga_K$ of a hyperbolic knot $K$ related to 
a triangulation of its $A$-polynomial curve?
\end{question}

To formulate our next question, recall that the tropical curve $\Ga_K$
is dual to a Newton subdivision of the 2-dimensional Newton polytope 
of the polynomial $A_K(M,L,q)$ with respect to the variables $L$ and $M$.
Assuming that $A_K(M,L,q)$ is good, and assuming the AJ Conjecture, it follows
that $\Ga_K$ is dual to the Newton polygon of the $A$-polynomial of $K$.
$\Ga_K$ is a balanced rational graph that consists or edges and rays, and the 
above assumptions imply that the slopes of the rays are negative inverses
of the slopes of the $A$-polynomial of $K$. Consequently, Culler-Shalen theory
(see \cite{CS})
implies that the slopes of the rays of $\Ga_K$ are negative inverses
of boundary slopes of $K$, appropriately normalized by a factor of $2$.

\begin{question}
\lbl{que.2}
What is the geometric meaning of the vertices of $\Ga_K$ (those are
points in $\BQ^2$) and of the slopes of the edges of $\Ga_K$?
\end{question}

\section{Computations of tropical curves of knots}
\lbl{sec.compute}

\subsection{The homogeneous tropical curve of the $4_1$ knot}
\lbl{sub.41}

The non-commutative $A$-polynomial $A_{4_1}(M,L,q)$ of $4_1$ was computed 
in \cite[Sec.6.2]{GL} and also \cite[Sec.3.2]{Ga2} 
using the {\tt WZ method} of \cite{WZ,Z} implemented by \cite{PR}
in {\tt Mathematica}. The non-commutative $A$-polynomial is given by
%%%% see Mathematica file: TrefoilFigure8.nb 

{\small
\begin{eqnarray*}
A_{4_1}(y,x,1/t) &=&
\frac{x^3 \left(t^2-y\right) \left(t^3-y\right) y^2 (t+y) \left(t-y^2\right) 
\left(t^3-y^2\right)}{t^{14}} \\
&+ & 
\frac{\left(t^2-y\right) (-1+y) y^2 \left(t^2+y\right) \left(t^3-y^2\right) 
\left(t^5-y^2\right)}{t^{15}} \\
&- & 
\frac{1}{t^{18}}x^2 \left(t^2-y\right)^2 \left(t^2+y\right) 
\left(t-y^2\right) \left(t^3-y^2\right) 
\left(t^8-2 t^6 y+t^7 y-t^3 y^2+t^4 y^2-t^5 y^2+t y^3-2 t^2 y^3+y^4\right) \\
&+ & 
\frac{1}{t^{17}}x (t-y) \left(t^2-y\right) (t+y) \left(t^3-y^2\right) 
\left(t^5-y^2\right) 
\left(t^4+y^4-t^3 y (2+y)-t y^2 (1+2 y)+t^2 y \left(1+y+y^2\right)\right)
\end{eqnarray*}
}
Notice that 
$$
A_{4_1}(x,y,1)=(-1 + x) (-1 + y)^4 (1 + y)^3 
(-x + x y + y^2 + 2 x y^2 + x^2 y^2 +    x y^3 - x y^4)
$$
confirms the AJ Conjecture, since the last factor is the geometric component
of the $A$-polynomial of $4_1$, the first term is the abelian component
of the $A$-polynomial, and the remaining second and third terms depend only
on $y=M$. Expanding out the terms, we obtain that:

      \begin{math}
 A_{4_1}(y,x,1/t)=         \tfrac{1}{t^{18}}\cdot x^{2}y^{11}+\tfrac{-1}{t^{14}}\cdot x^{3}y^{9}+\tfrac{1-3\cdot t}{t^{17}}\cdot x^{2}y^{10}+\tfrac{-1}{t^{17}}\cdot xy^{11}+\tfrac{-1+t+t^{2}}{t^{13}}\cdot x^{3}y^{8}+\tfrac{-1-3\cdot t^{2}+2\cdot t^{3}-t^{4}}{t^{17}}\cdot x^{2}y^{9}+\tfrac{2}{t^{16}}\cdot xy^{10}+\tfrac{1+2\cdot t^{2}+t^{3}-t^{4}}{t^{13}}\cdot x^{3}y^{7}+\tfrac{-1+3\cdot t-t^{2}+3\cdot t^{3}+2\cdot t^{5}}{t^{16}}\cdot x^{2}y^{8}+\tfrac{1+t^{3}+t^{4}}{t^{16}}\cdot xy^{9}+\tfrac{1-t-t^{3}-2\cdot t^{4}}{t^{12}}\cdot x^{3}y^{6}+\tfrac{3-2\cdot t+3\cdot t^{2}+t^{4}}{t^{14}}\cdot x^{2}y^{7}+\tfrac{-1-t-t^{2}-t^{3}-2\cdot t^{4}}{t^{15}}\cdot xy^{8}+\tfrac{1}{t^{15}}\cdot y^{9}+\tfrac{-2-t-t^{3}+t^{4}}{t^{10}}\cdot x^{3}y^{5}+\tfrac{1-3\cdot t-2\cdot t^{3}-t^{4}-2\cdot t^{6}}{t^{13}}\cdot x^{2}y^{6}+\tfrac{-1-t^{2}-t^{3}-t^{4}-t^{6}}{t^{14}}\cdot xy^{7}+\tfrac{-1}{t^{15}}\cdot y^{8}+\tfrac{-1+t+2\cdot t^{2}+t^{4}}{t^{9}}\cdot x^{3}y^{4}+\tfrac{-2-t^{2}-2\cdot t^{3}-3\cdot t^{5}+t^{6}}{t^{11}}\cdot x^{2}y^{5}+\tfrac{1+t^{2}+t^{3}+2\cdot t^{4}+t^{5}+t^{6}}{t^{13}}\cdot xy^{6}+\tfrac{-1-t-t^{2}}{t^{12}}\cdot y^{7}+\tfrac{1+t-t^{2}}{t^{7}}\cdot x^{3}y^{3}+\tfrac{1+3\cdot t^{2}-2\cdot t^{3}+3\cdot t^{4}}{t^{8}}\cdot x^{2}y^{4}+\tfrac{1+t+2\cdot t^{2}+t^{3}+t^{4}+t^{6}}{t^{11}}\cdot xy^{5}+\tfrac{1+t+t^{2}}{t^{12}}\cdot y^{6}+\tfrac{-1}{t^{4}}\cdot x^{3}y^{2}+\tfrac{2+3\cdot t^{2}-t^{3}+3\cdot t^{4}-t^{5}}{t^{7}}\cdot x^{2}y^{3}+\tfrac{-1-t^{2}-t^{3}-t^{4}-t^{6}}{t^{10}}\cdot xy^{4}+\tfrac{1+t+t^{2}}{t^{8}}\cdot y^{5}+\tfrac{-1+2\cdot t-3\cdot t^{2}-t^{4}}{t^{5}}\cdot x^{2}y^{2}+\tfrac{-2-t-t^{2}-t^{3}-t^{4}}{t^{7}}\cdot xy^{3}+\tfrac{-1-t-t^{2}}{t^{8}}\cdot y^{4}+\tfrac{-3+t}{t^{2}}\cdot x^{2}y+\tfrac{1+t+t^{4}}{t^{6}}\cdot xy^{2}+\tfrac{-1}{t^{3}}\cdot y^{3}+x^{2}+\tfrac{2}{t^{2}}\cdot xy+\tfrac{1}{t^{3}}\cdot y^{2}+\tfrac{-1}{t}\cdot x
      \end{math}

Inspection of the above formula shows that $A_{4_1}(y,x,1/t)$ is good.
Using the drawing {\tt polymake}
program of \cite{Ma} implemented in {\tt Singular} one can
compute the vertices of the tropical curve:
%%%%%% see file: notes/APoly41.tex computed by polymake 

{\small
   \begin{displaymath}
      (3,-1/2),\;\;(-1,-1/3),\;\;(-3/4,-1/2),\;\;(-2,0),\;\;(2,-1),
\;\;(-1/2,-1),\;\;(1,-3/2),\;\;(0,-3/2),\;\;
   \end{displaymath}
   \begin{displaymath}
(-1/2,-5/4),\;\;(1/2,-7/4),\;\;(-1,-3/2),\;\;(1/2,-2),\;\;(2,-3),
\;\;(3/4,-5/2),\;\;(1,-8/3),\;\;(-2,-2),\;\;(-3,-5/2)
   \end{displaymath}
}

The tropical curve (with the convention that unmarked edges
or rays have multiplicity $1$) is:
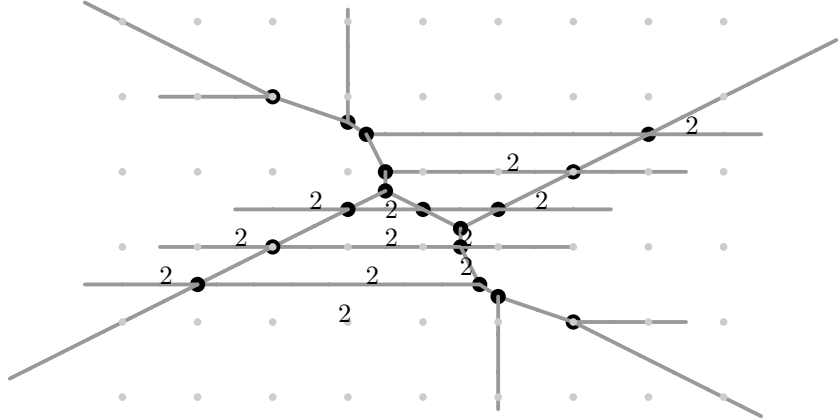
\vspace*{0.5cm}

   \begin{center}
    \begin{texdraw}
       \drawdim cm  \relunitscale 0.5 \arrowheadtype t:V
%       \drawdim cm  \relunitscale 0.7 \arrowheadtype t:V
       %\linewd 0.05 \lpatt (0.1 0.4)
       %\move (-4 0) \avec (9 0) \move (0 -4) \avec (0 9)
       \linewd 0.1  \lpatt (1 0)

       \setgray 0.6
       \relunitscale 2
       \move (3 0.5) \fcir f:0 r:0.1
       \move (3 0.5) \lvec (-0.75 0.5)
       \htext (1.12 0){$2$}
       \move (3 0.5) \lvec (2 0)
       \move (3 0.5) \rlvec (2.5 1.25)
       \move (3 0.5) \rlvec (1.5 0)
       \htext (3.5 0.5){$2$}
       \move (-1 0.66) \fcir f:0 r:0.1
       \move (-1 0.66) \lvec (-0.75 0.5)
       \move (-1 0.66) \lvec (-2 1)
       \move (-1 0.66) \rlvec (0 1.5)
       \move (-0.75 0.5) \fcir f:0 r:0.1
       \move (-0.75 0.5) \lvec (-0.5 0)
       \move (-2 1) \fcir f:0 r:0.1
       \move (-2 1) \rlvec (-1.5 0)
       \move (-2 1) \rlvec (-2.5 1.25)
       \move (2 0) \fcir f:0 r:0.1
       \move (2 0) \lvec (-0.5 0)
       \move (2 0) \lvec (1 -0.5)
       \move (2 0) \rlvec (1.5 0)
       \move (-0.5 0) \fcir f:0 r:0.1
       \move (-0.5 0) \lvec (-0.5 -0.25)
       \htext (-0.5 -0.62){$2$}
       \move (1 -0.5) \fcir f:0 r:0.1
       \move (1 -0.5) \lvec (0 -0.5)
       \htext (0.5 -1){$2$}
       \move (1 -0.5) \lvec (0.5 -0.75)
       \move (1 -0.5) \rlvec (1.5 0)
       \htext (1.5 -0.5){$2$}
       \move (0 -0.5) \fcir f:0 r:0.1
       \move (0 -0.5) \lvec (-0.5 -0.25)
       \move (0 -0.5) \lvec (0.5 -0.75)
       \move (0 -0.5) \lvec (-1 -0.5)
       \htext (-0.5 -1){$2$}
       \move (-0.5 -0.25) \fcir f:0 r:0.1
       \move (-0.5 -0.25) \lvec (-1 -0.5)
       \move (0.5 -0.75) \fcir f:0 r:0.1
       \move (0.5 -0.75) \lvec (0.5 -1)
       \htext (0.5 -1.37){$2$}
       \move (-1 -0.5) \fcir f:0 r:0.1
       \move (-1 -0.5) \lvec (-2 -1)
       \move (-1 -0.5) \rlvec (-1.5 0)
       \htext (-1.5 -0.5){$2$}
       \move (0.5 -1) \fcir f:0 r:0.1
       \move (0.5 -1) \lvec (0.75 -1.5)
       \move (0.5 -1) \lvec (-2 -1)
       \htext (-0.75 -1.5){$2$}
       \move (0.5 -1) \rlvec (1.5 0)
       \move (2 -2) \fcir f:0 r:0.1
       \move (2 -2) \lvec (1 -1.66)
       \move (2 -2) \rlvec (1.5 0)
       \move (2 -2) \rlvec (2.5 -1.25)
       \move (0.75 -1.5) \fcir f:0 r:0.1
       \move (0.75 -1.5) \lvec (1 -1.66)
       \move (0.75 -1.5) \lvec (-3 -1.5)
       \htext (-1.12 -2){$2$}
       \move (1 -1.66) \fcir f:0 r:0.1
       \move (1 -1.66) \rlvec (0 -1.5)
       \move (-2 -1) \fcir f:0 r:0.1
       \move (-2 -1) \lvec (-3 -1.5)
       \move (-2 -1) \rlvec (-1.5 0)
       \htext (-2.5 -1){$2$}
       \move (-3 -1.5) \fcir f:0 r:0.1
       \move (-3 -1.5) \rlvec (-2.5 -1.25)
       \move (-3 -1.5) \rlvec (-1.5 0)
       \htext (-3.5 -1.5){$2$}

   %% HERE STARTS THE CODE FOR THE LATTICE
        \move (-4 -3) \fcir f:0.8 r:0.05
        \move (-4 -2) \fcir f:0.8 r:0.05
        \move (-4 -1) \fcir f:0.8 r:0.05
        \move (-4 0) \fcir f:0.8 r:0.05
        \move (-4 1) \fcir f:0.8 r:0.05
        \move (-4 2) \fcir f:0.8 r:0.05
        \move (-3 -3) \fcir f:0.8 r:0.05
        \move (-3 -2) \fcir f:0.8 r:0.05
        \move (-3 -1) \fcir f:0.8 r:0.05
        \move (-3 0) \fcir f:0.8 r:0.05
        \move (-3 1) \fcir f:0.8 r:0.05
        \move (-3 2) \fcir f:0.8 r:0.05
        \move (-2 -3) \fcir f:0.8 r:0.05
        \move (-2 -2) \fcir f:0.8 r:0.05
        \move (-2 -1) \fcir f:0.8 r:0.05
        \move (-2 0) \fcir f:0.8 r:0.05
        \move (-2 1) \fcir f:0.8 r:0.05
        \move (-2 2) \fcir f:0.8 r:0.05
        \move (-1 -3) \fcir f:0.8 r:0.05
        \move (-1 -2) \fcir f:0.8 r:0.05
        \move (-1 -1) \fcir f:0.8 r:0.05
        \move (-1 0) \fcir f:0.8 r:0.05
        \move (-1 1) \fcir f:0.8 r:0.05
        \move (-1 2) \fcir f:0.8 r:0.05
        \move (0 -3) \fcir f:0.8 r:0.05
        \move (0 -2) \fcir f:0.8 r:0.05
        \move (0 -1) \fcir f:0.8 r:0.05
        \move (0 0) \fcir f:0.8 r:0.05
        \move (0 1) \fcir f:0.8 r:0.05
        \move (0 2) \fcir f:0.8 r:0.05
        \move (1 -3) \fcir f:0.8 r:0.05
        \move (1 -2) \fcir f:0.8 r:0.05
        \move (1 -1) \fcir f:0.8 r:0.05
        \move (1 0) \fcir f:0.8 r:0.05
        \move (1 1) \fcir f:0.8 r:0.05
        \move (1 2) \fcir f:0.8 r:0.05
        \move (2 -3) \fcir f:0.8 r:0.05
        \move (2 -2) \fcir f:0.8 r:0.05
        \move (2 -1) \fcir f:0.8 r:0.05
        \move (2 0) \fcir f:0.8 r:0.05
        \move (2 1) \fcir f:0.8 r:0.05
        \move (2 2) \fcir f:0.8 r:0.05
        \move (3 -3) \fcir f:0.8 r:0.05
        \move (3 -2) \fcir f:0.8 r:0.05
        \move (3 -1) \fcir f:0.8 r:0.05
        \move (3 0) \fcir f:0.8 r:0.05
        \move (3 1) \fcir f:0.8 r:0.05
        \move (3 2) \fcir f:0.8 r:0.05
        \move (4 -3) \fcir f:0.8 r:0.05
        \move (4 -2) \fcir f:0.8 r:0.05
        \move (4 -1) \fcir f:0.8 r:0.05
        \move (4 0) \fcir f:0.8 r:0.05
        \move (4 1) \fcir f:0.8 r:0.05
        \move (4 2) \fcir f:0.8 r:0.05
   %% HERE ENDS THE CODE FOR THE LATTICE
                          
    \end{texdraw}\end{center}
The Newton subdivision of the Newton polygon is:
%   \vspace*{0.5cm}

   
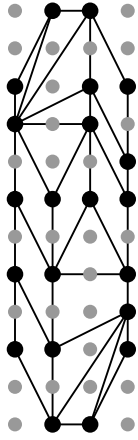
\begin{center}
    \begin{texdraw}
       \drawdim cm  \relunitscale 0.5 % original value is: 1 
       \linewd 0.05
        \move (2 11)        
        \lvec (3 9)
        \move (3 9)        
        \lvec (3 2)
        \move (3 2)        
        \lvec (2 0)
        \move (2 0)        
        \lvec (1 0)
        \move (1 0)        
        \lvec (0 2)
        \move (0 2)        
        \lvec (0 9)
        \move (0 9)        
        \lvec (1 11)
        \move (1 11)        
        \lvec (2 11)

        \move (2 9)        
        \lvec (2 11)
        \move (3 7)        
        \lvec (2 9)
        \move (2 11)        
        \lvec (0 8)
        \move (0 8)        
        \lvec (1 11)
        \move (2 9)        
        \lvec (0 8)
        \move (2 8)        
        \lvec (2 9)
        \move (3 6)        
        \lvec (2 8)
        \move (2 8)        
        \lvec (0 8)
        \move (2 6)        
        \lvec (2 8)
        \move (3 4)        
        \lvec (2 6)
        \move (1 6)        
        \lvec (2 8)
        \move (2 6)        
        \lvec (1 4)
        \move (1 4)        
        \lvec (1 6)
        \move (1 6)        
        \lvec (0 8)
        \move (3 4)        
        \lvec (1 4)
        \move (1 4)        
        \lvec (0 6)
        \move (3 3)        
        \lvec (1 2)
        \move (1 2)        
        \lvec (1 4)
        \move (2 0)        
        \lvec (3 3)
        \move (3 3)        
        \lvec (1 0)
        \move (1 0)        
        \lvec (1 2)
        \move (1 2)        
        \lvec (0 4)
        \move (0 0) \fcir f:0.6 r:0.18
        \move (0 1) \fcir f:0.6 r:0.18
        \move (0 2) \fcir f:0.6 r:0.18
        \move (0 3) \fcir f:0.6 r:0.18
        \move (0 4) \fcir f:0.6 r:0.18
        \move (0 5) \fcir f:0.6 r:0.18
        \move (0 6) \fcir f:0.6 r:0.18
        \move (0 7) \fcir f:0.6 r:0.18
        \move (0 8) \fcir f:0.6 r:0.18
        \move (0 9) \fcir f:0.6 r:0.18
        \move (0 10) \fcir f:0.6 r:0.18
        \move (0 11) \fcir f:0.6 r:0.18
        \move (1 0) \fcir f:0.6 r:0.18
        \move (1 1) \fcir f:0.6 r:0.18
        \move (1 2) \fcir f:0.6 r:0.18
        \move (1 3) \fcir f:0.6 r:0.18
        \move (1 4) \fcir f:0.6 r:0.18
        \move (1 5) \fcir f:0.6 r:0.18
        \move (1 6) \fcir f:0.6 r:0.18
        \move (1 7) \fcir f:0.6 r:0.18
        \move (1 8) \fcir f:0.6 r:0.18
        \move (1 9) \fcir f:0.6 r:0.18
        \move (1 10) \fcir f:0.6 r:0.18
        \move (1 11) \fcir f:0.6 r:0.18
        \move (2 0) \fcir f:0.6 r:0.18
        \move (2 1) \fcir f:0.6 r:0.18
        \move (2 2) \fcir f:0.6 r:0.18
        \move (2 3) \fcir f:0.6 r:0.18
        \move (2 4) \fcir f:0.6 r:0.18
        \move (2 5) \fcir f:0.6 r:0.18
        \move (2 6) \fcir f:0.6 r:0.18
        \move (2 7) \fcir f:0.6 r:0.18
        \move (2 8) \fcir f:0.6 r:0.18
        \move (2 9) \fcir f:0.6 r:0.18
        \move (2 10) \fcir f:0.6 r:0.18
        \move (2 11) \fcir f:0.6 r:0.18
        \move (3 0) \fcir f:0.6 r:0.18
        \move (3 1) \fcir f:0.6 r:0.18
        \move (3 2) \fcir f:0.6 r:0.18
        \move (3 3) \fcir f:0.6 r:0.18
        \move (3 4) \fcir f:0.6 r:0.18
        \move (3 5) \fcir f:0.6 r:0.18
        \move (3 6) \fcir f:0.6 r:0.18
        \move (3 7) \fcir f:0.6 r:0.18
        \move (3 8) \fcir f:0.6 r:0.18
        \move (3 9) \fcir f:0.6 r:0.18
        \move (3 10) \fcir f:0.6 r:0.18
        \move (3 11) \fcir f:0.6 r:0.18
       \move (2 11) 
       \fcir f:0 r:0.22
       \move (3 9) 
       \fcir f:0 r:0.22
       \move (2 9) 
       \fcir f:0 r:0.22
       \move (3 7) 
       \fcir f:0 r:0.22
       \move (1 11) 
       \fcir f:0 r:0.22
       \move (0 8) 
       \fcir f:0 r:0.22
       \move (0 9) 
       \fcir f:0 r:0.22
       \move (2 8) 
       \fcir f:0 r:0.22
       \move (3 6) 
       \fcir f:0 r:0.22
       \move (2 6) 
       \fcir f:0 r:0.22
       \move (3 4) 
       \fcir f:0 r:0.22
       \move (1 6) 
       \fcir f:0 r:0.22
       \move (1 4) 
       \fcir f:0 r:0.22
       \move (0 6) 
       \fcir f:0 r:0.22
       \move (3 3) 
       \fcir f:0 r:0.22
       \move (1 2) 
       \fcir f:0 r:0.22
       \move (3 2) 
       \fcir f:0 r:0.22
       \move (2 0) 
       \fcir f:0 r:0.22
       \move (1 0) 
       \fcir f:0 r:0.22
       \move (0 4) 
       \fcir f:0 r:0.22
       \move (0 2) 
       \fcir f:0 r:0.22
   \end{texdraw}\end{center}
The reader may observe that the above Newton polygon is the Minkowski
sum of the Newton polygon of the $A$-polynomial of $4_1$ with 
a vertical segment.

\subsection{The non-homogeneous tropical curve of the $4_1$ knot}
\lbl{sub.41non}

The non-homogeneous $A$-polynomial of the $4_1$ knot was computed in 
Theorem 1 of \cite{GS} (with the notation $A_{-1}(E,Q,q)$ where $E=L$ and 
$Q=M$). It has $22$ terms and it is given by:
%%%% see mathematica file: sun/ApolyTwistKnots.nb 
%%%% see mathematica file: sun/AqPolyTwistKnots.nb
\begin{eqnarray*}
A^{nh}_{4_1}(M,L,q)&=&
L^2 M^2 q^2 \left(-1+M^2 q\right) \left(-1+M q^2\right) 
+(-1+M) M^2 q^2 \left(-1+M^2 q^3\right)
\\ & & 
-L (-1+M q)^2 (1+M q) \left(1-M q-M^2 q-M^2 q^3-M^3 q^3+M^4 q^4\right)
\\
B_{4_1}(M,L)&=& M q (1+M q) \left(-1+M^2 q\right) \left(-1+M^2 q^3\right)
\end{eqnarray*}
It follows that: 
%%%% see polymake file: notes/AnhPoly41.tex

      \begin{math}
A^{nh}_{4_1}(y,x,1/t)=\tfrac{-1}{t^{7}}\cdot xy^{7}+\tfrac{1}{t^{5}}\cdot x^{2}y^{5}+\tfrac{2}{t^{6}}\cdot xy^{6}+\tfrac{-1}{t^{3}}\cdot x^{2}y^{4}+\tfrac{1+t^{2}}{t^{6}}\cdot xy^{5}+\tfrac{-1}{t^{4}}\cdot x^{2}y^{3}+\tfrac{-1-t-t^{2}}{t^{5}}\cdot xy^{4}+\tfrac{1}{t^{5}}\cdot y^{5}+\tfrac{1}{t^{2}}\cdot x^{2}y^{2}+\tfrac{-1-t-t^{2}}{t^{4}}\cdot xy^{3}+\tfrac{-1}{t^{5}}\cdot y^{4}+\tfrac{1+t^{2}}{t^{3}}\cdot xy^{2}+\tfrac{-1}{t^{2}}\cdot y^{3}+\tfrac{2}{t}\cdot xy+\tfrac{1}{t^{2}}\cdot y^{2}-x
      \end{math}

\noindent
It is easy to see that the above polynomial is good. 
The vertices of the corresponding tropical curve are:

   \begin{displaymath}
      (1,-1/2),\;\;(-1/2,-1/2),\;\;(-2,0),\;\;(0,-1),\;\;(2,-2),\;\;(1/2,-3/2),\;\;(-1,-3/2)
   \end{displaymath}
The tropical curve is:
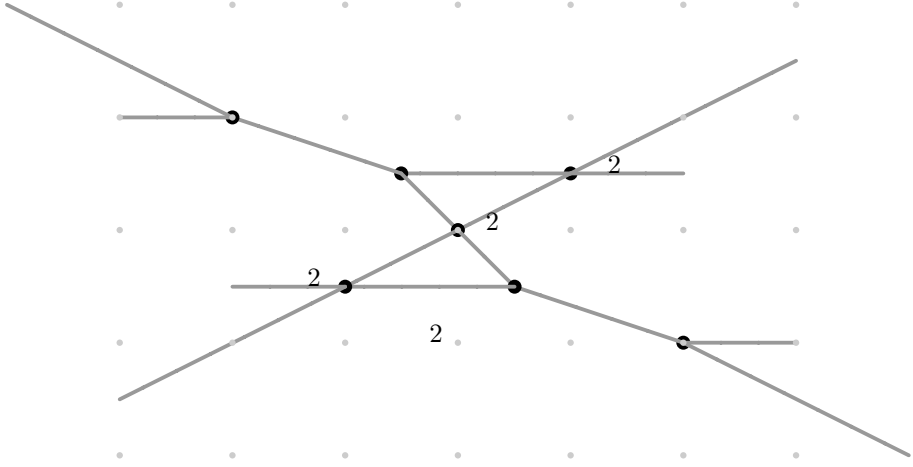
\vspace*{0.5cm}

\begin{center}
    \begin{texdraw}
       \drawdim cm  \relunitscale 0.5 \arrowheadtype t:V
%       \drawdim cm  \relunitscale 0.7 \arrowheadtype t:V
       %\linewd 0.05 \lpatt (0.1 0.4)
       %\move (-4 0) \avec (9 0) \move (0 -4) \avec (0 9)
       \linewd 0.1  \lpatt (1 0)

       \setgray 0.6
       \relunitscale 3
       \move (1 0.5) \fcir f:0 r:0.06
       \move (1 0.5) \lvec (-0.5 0.5)
       \htext (0.25 0){$2$}
       \move (1 0.5) \lvec (0 0)
       \move (1 0.5) \rlvec (2 1)
       \move (1 0.5) \rlvec (1 0)
       \htext (1.33 0.5){$2$}
       \move (-0.5 0.5) \fcir f:0 r:0.06
       \move (-0.5 0.5) \lvec (-2 1)
       \move (-0.5 0.5) \lvec (0 0)
       \move (-2 1) \fcir f:0 r:0.06
       \move (-2 1) \rlvec (-1 0)
       \move (-2 1) \rlvec (-2 1)
       \move (0 0) \fcir f:0 r:0.06
       \move (0 0) \lvec (0.5 -0.5)
       \move (0 0) \lvec (-1 -0.5)
       \move (2 -1) \fcir f:0 r:0.06
       \move (2 -1) \lvec (0.5 -0.5)
       \move (2 -1) \rlvec (1 0)
       \move (2 -1) \rlvec (2 -1)
       \move (0.5 -0.5) \fcir f:0 r:0.06
       \move (0.5 -0.5) \lvec (-1 -0.5)
       \htext (-0.25 -1){$2$}
       \move (-1 -0.5) \fcir f:0 r:0.06
       \move (-1 -0.5) \rlvec (-2 -1)
       \move (-1 -0.5) \rlvec (-1 0)
       \htext (-1.33 -0.5){$2$}

   %% HERE STARTS THE CODE FOR THE LATTICE
        \move (-3 -2) \fcir f:0.8 r:0.03
        \move (-3 -1) \fcir f:0.8 r:0.03
        \move (-3 0) \fcir f:0.8 r:0.03
        \move (-3 1) \fcir f:0.8 r:0.03
        \move (-3 2) \fcir f:0.8 r:0.03
        \move (-2 -2) \fcir f:0.8 r:0.03
        \move (-2 -1) \fcir f:0.8 r:0.03
        \move (-2 0) \fcir f:0.8 r:0.03
        \move (-2 1) \fcir f:0.8 r:0.03
        \move (-2 2) \fcir f:0.8 r:0.03
        \move (-1 -2) \fcir f:0.8 r:0.03
        \move (-1 -1) \fcir f:0.8 r:0.03
        \move (-1 0) \fcir f:0.8 r:0.03
        \move (-1 1) \fcir f:0.8 r:0.03
        \move (-1 2) \fcir f:0.8 r:0.03
        \move (0 -2) \fcir f:0.8 r:0.03
        \move (0 -1) \fcir f:0.8 r:0.03
        \move (0 0) \fcir f:0.8 r:0.03
        \move (0 1) \fcir f:0.8 r:0.03
        \move (0 2) \fcir f:0.8 r:0.03
        \move (1 -2) \fcir f:0.8 r:0.03
        \move (1 -1) \fcir f:0.8 r:0.03
        \move (1 0) \fcir f:0.8 r:0.03
        \move (1 1) \fcir f:0.8 r:0.03
        \move (1 2) \fcir f:0.8 r:0.03
        \move (2 -2) \fcir f:0.8 r:0.03
        \move (2 -1) \fcir f:0.8 r:0.03
        \move (2 0) \fcir f:0.8 r:0.03
        \move (2 1) \fcir f:0.8 r:0.03
        \move (2 2) \fcir f:0.8 r:0.03
        \move (3 -2) \fcir f:0.8 r:0.03
        \move (3 -1) \fcir f:0.8 r:0.03
        \move (3 0) \fcir f:0.8 r:0.03
        \move (3 1) \fcir f:0.8 r:0.03
        \move (3 2) \fcir f:0.8 r:0.03
   %% HERE ENDS THE CODE FOR THE LATTICE
                          
    \end{texdraw}\end{center}
\noindent
The Newton subdivision of the Newton polygon is:

   
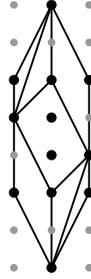
\begin{center}
    \begin{texdraw}
       \drawdim cm  \relunitscale 0.5 %%% original value: 1
       \linewd 0.05
        \move (1 7)        
        \lvec (2 5)
        \move (2 5)        
        \lvec (2 2)
        \move (2 2)        
        \lvec (1 0)
        \move (1 0)        
        \lvec (0 2)
        \move (0 2)        
        \lvec (0 5)
        \move (0 5)        
        \lvec (1 7)

        \move (1 5)        
        \lvec (1 7)
        \move (2 3)        
        \lvec (1 5)
        \move (0 4)        
        \lvec (1 7)
        \move (1 5)        
        \lvec (0 4)
        \move (2 3)        
        \lvec (1 2)
        \move (1 2)        
        \lvec (0 4)
        \move (1 0)        
        \lvec (2 3)
        \move (1 0)        
        \lvec (1 2)
        \move (0 0) \fcir f:0.6 r:0.11
        \move (0 1) \fcir f:0.6 r:0.11
        \move (0 2) \fcir f:0.6 r:0.11
        \move (0 3) \fcir f:0.6 r:0.11
        \move (0 4) \fcir f:0.6 r:0.11
        \move (0 5) \fcir f:0.6 r:0.11
        \move (0 6) \fcir f:0.6 r:0.11
        \move (0 7) \fcir f:0.6 r:0.11
        \move (1 0) \fcir f:0.6 r:0.11
        \move (1 1) \fcir f:0.6 r:0.11
        \move (1 2) \fcir f:0.6 r:0.11
        \move (1 3) \fcir f:0.6 r:0.11
        \move (1 4) \fcir f:0.6 r:0.11
        \move (1 5) \fcir f:0.6 r:0.11
        \move (1 6) \fcir f:0.6 r:0.11
        \move (1 7) \fcir f:0.6 r:0.11
        \move (2 0) \fcir f:0.6 r:0.11
        \move (2 1) \fcir f:0.6 r:0.11
        \move (2 2) \fcir f:0.6 r:0.11
        \move (2 3) \fcir f:0.6 r:0.11
        \move (2 4) \fcir f:0.6 r:0.11
        \move (2 5) \fcir f:0.6 r:0.11
        \move (2 6) \fcir f:0.6 r:0.11
        \move (2 7) \fcir f:0.6 r:0.11
       \move (1 7) 
       \fcir f:0 r:0.14
       \move (2 5) 
       \fcir f:0 r:0.14
       \move (1 5) 
       \fcir f:0 r:0.14
       \move (2 3) 
       \fcir f:0 r:0.14
       \move (0 4) 
       \fcir f:0 r:0.14
       \move (0 5) 
       \fcir f:0 r:0.14
       \move (1 4) 
       \fcir f:0 r:0.14
       \move (1 3) 
       \fcir f:0 r:0.14
       \move (1 2) 
       \fcir f:0 r:0.14
       \move (2 2) 
       \fcir f:0 r:0.14
       \move (1 0) 
       \fcir f:0 r:0.14
       \move (0 2) 
       \fcir f:0 r:0.14
   \end{texdraw}
   \end{center}

\noindent
This example exhibits that the non-homogeneous tropical curve is much simpler
than the homogeneous one.

\subsection{The non-homogeneous tropical curve of the $5_2$ knot}
\lbl{sub.52non}

The non-homogeneous non-commutative $A$-polynomial $A^{nh}_{5_2}(M,L,q)$ has
$98$ terms, and it is given by \cite{GS} (with the notation $A^{nh}_2(E,Q,q)$
where $E=L$, $Q=M$):

      \begin{math}
A^{nh}_{5_2}(y,x,1/t)          =\tfrac{-1}{t^{19}}\cdot xy^{12}+\tfrac{1}{t^{17}}\cdot x^{2}y^{10}+\tfrac{3}{t^{18}}\cdot xy^{11}+\tfrac{1}{t^{18}}\cdot y^{12}+\tfrac{-2}{t^{15}}\cdot x^{2}y^{9}+\tfrac{1-t+t^{3}+t^{4}}{t^{18}}\cdot xy^{10}+\tfrac{-1}{t^{18}}\cdot y^{11}+\tfrac{-1-t+t^{2}+t^{3}-t^{4}}{t^{16}}\cdot x^{2}y^{8}+\tfrac{-2-2\cdot t+t^{2}-2\cdot t^{3}-3\cdot t^{4}}{t^{17}}\cdot xy^{9}+\tfrac{-1-t}{t^{14}}\cdot y^{10}+\tfrac{2+2\cdot t-t^{2}+t^{3}+2\cdot t^{4}}{t^{14}}\cdot x^{2}y^{7}+\tfrac{1-2\cdot t-t^{2}+t^{3}-t^{5}}{t^{15}}\cdot xy^{8}+\tfrac{1+t}{t^{14}}\cdot y^{9}+\tfrac{1}{t^{6}}\cdot x^{3}y^{5}+\tfrac{1-t-t^{2}+t^{3}+t^{4}-2\cdot t^{5}}{t^{14}}\cdot x^{2}y^{6}+\tfrac{2-t+t^{2}+4\cdot t^{3}+2\cdot t^{4}-t^{5}+2\cdot t^{6}}{t^{15}}\cdot xy^{7}+\tfrac{1}{t^{9}}\cdot y^{8}+\tfrac{-1}{t^{3}}\cdot x^{3}y^{4}+\tfrac{-2+t-t^{2}-4\cdot t^{3}-2\cdot t^{4}+t^{5}-2\cdot t^{6}}{t^{12}}\cdot x^{2}y^{5}+\tfrac{-1+t+t^{2}-t^{3}-t^{4}+2\cdot t^{5}}{t^{14}}\cdot xy^{6}+\tfrac{-1}{t^{9}}\cdot y^{7}+\tfrac{-1-t}{t^{5}}\cdot x^{3}y^{3}+\tfrac{-1+2\cdot t+t^{2}-t^{3}+t^{5}}{t^{9}}\cdot x^{2}y^{4}+\tfrac{-2-2\cdot t+t^{2}-t^{3}-2\cdot t^{4}}{t^{11}}\cdot xy^{5}+\tfrac{1+t}{t^{2}}\cdot x^{3}y^{2}+\tfrac{2+2\cdot t-t^{2}+2\cdot t^{3}+3\cdot t^{4}}{t^{8}}\cdot x^{2}y^{3}+\tfrac{1+t-t^{2}-t^{3}+t^{4}}{t^{10}}\cdot xy^{4}+\tfrac{1}{t^{3}}\cdot x^{3}y+\tfrac{-1+t-t^{3}-t^{4}}{t^{6}}\cdot x^{2}y^{2}+\tfrac{2}{t^{6}}\cdot xy^{3}-x^{3}+\tfrac{-3}{t^{3}}\cdot x^{2}y+\tfrac{-1}{t^{5}}\cdot xy^{2}+\tfrac{1}{t}\cdot x^{2}
      \end{math}

\noindent
The vertices of the tropical curve are:
   \begin{displaymath}
      (1,-1/2),\;\;(-1,0),\;\;(-1/2,-1/2),\;\;(17/2,-1/2),\;\;(-1,-1),\;\;(0,-1),\;\;(-6,-2),\;\;
    \end{displaymath}
    \begin{displaymath}
(6,-1),\;\;(-17/2,-5/2),\;\;(0,-2),\;\;(1,-2),\;\;(-1,-5/2),\;\;(1/2,-5/2),\;\;(1,-3)
   \end{displaymath}
   \vspace*{0.5cm}

\noindent
   \vspace*{0.5cm}
   The Newton subdivision of the tropical curve is:
   \vspace*{0.5cm}

   \begin{center}
            
    \begin{texdraw}
       \drawdim cm  \relunitscale 0.5 %% originally: 1
       \linewd 0.05
        \move (1 12)        
        \lvec (2 10)
        \move (2 10)        
        \lvec (3 5)
        \move (3 5)        
        \lvec (3 0)
        \move (3 0)        
        \lvec (2 0)
        \move (2 0)        
        \lvec (1 2)
        \move (1 2)        
        \lvec (0 7)
        \move (0 7)        
        \lvec (0 12)
        \move (0 12)        
        \lvec (1 12)

        \move (1 10)        
        \lvec (1 12)
        \move (2 10)        
        \lvec (2 8)
        \move (2 8)        
        \lvec (1 10)
        \move (1 12)        
        \lvec (0 11)
        \move (1 10)        
        \lvec (0 11)
        \move (3 3)        
        \lvec (2 8)
        \move (1 7)        
        \lvec (1 6)
        \move (1 9)        
        \lvec (1 7)
        \move (1 10)        
        \lvec (1 9)
        \move (1 6)        
        \lvec (0 11)
        \move (2 8)        
        \lvec (2 6)
        \move (2 6)        
        \lvec (1 6)
        \move (1 4)        
        \lvec (0 9)
        \move (1 6)        
        \lvec (1 4)
        \move (3 1)        
        \lvec (2 6)
        \move (1 4)        
        \lvec (1 2)
        \move (2 3)        
        \lvec (2 2)
        \move (2 5)        
        \lvec (2 3)
        \move (2 6)        
        \lvec (2 5)
        \move (2 2)        
        \lvec (1 4)
        \move (3 1)        
        \lvec (2 2)
        \move (2 2)        
        \lvec (2 0)
        \move (3 1)        
        \lvec (2 0)
        \move (0 0) \fcir f:0.6 r:0.2
        \move (0 1) \fcir f:0.6 r:0.2
        \move (0 2) \fcir f:0.6 r:0.2
        \move (0 3) \fcir f:0.6 r:0.2
        \move (0 4) \fcir f:0.6 r:0.2
        \move (0 5) \fcir f:0.6 r:0.2
        \move (0 6) \fcir f:0.6 r:0.2
        \move (0 7) \fcir f:0.6 r:0.2
        \move (0 8) \fcir f:0.6 r:0.2
        \move (0 9) \fcir f:0.6 r:0.2
        \move (0 10) \fcir f:0.6 r:0.2
        \move (0 11) \fcir f:0.6 r:0.2
        \move (0 12) \fcir f:0.6 r:0.2
        \move (1 0) \fcir f:0.6 r:0.2
        \move (1 1) \fcir f:0.6 r:0.2
        \move (1 2) \fcir f:0.6 r:0.2
        \move (1 3) \fcir f:0.6 r:0.2
        \move (1 4) \fcir f:0.6 r:0.2
        \move (1 5) \fcir f:0.6 r:0.2
        \move (1 6) \fcir f:0.6 r:0.2
        \move (1 7) \fcir f:0.6 r:0.2
        \move (1 8) \fcir f:0.6 r:0.2
        \move (1 9) \fcir f:0.6 r:0.2
        \move (1 10) \fcir f:0.6 r:0.2
        \move (1 11) \fcir f:0.6 r:0.2
        \move (1 12) \fcir f:0.6 r:0.2
        \move (2 0) \fcir f:0.6 r:0.2
        \move (2 1) \fcir f:0.6 r:0.2
        \move (2 2) \fcir f:0.6 r:0.2
        \move (2 3) \fcir f:0.6 r:0.2
        \move (2 4) \fcir f:0.6 r:0.2
        \move (2 5) \fcir f:0.6 r:0.2
        \move (2 6) \fcir f:0.6 r:0.2
        \move (2 7) \fcir f:0.6 r:0.2
        \move (2 8) \fcir f:0.6 r:0.2
        \move (2 9) \fcir f:0.6 r:0.2
        \move (2 10) \fcir f:0.6 r:0.2
        \move (2 11) \fcir f:0.6 r:0.2
        \move (2 12) \fcir f:0.6 r:0.2
        \move (3 0) \fcir f:0.6 r:0.2
        \move (3 1) \fcir f:0.6 r:0.2
        \move (3 2) \fcir f:0.6 r:0.2
        \move (3 3) \fcir f:0.6 r:0.2
        \move (3 4) \fcir f:0.6 r:0.2
        \move (3 5) \fcir f:0.6 r:0.2
        \move (3 6) \fcir f:0.6 r:0.2
        \move (3 7) \fcir f:0.6 r:0.2
        \move (3 8) \fcir f:0.6 r:0.2
        \move (3 9) \fcir f:0.6 r:0.2
        \move (3 10) \fcir f:0.6 r:0.2
        \move (3 11) \fcir f:0.6 r:0.2
        \move (3 12) \fcir f:0.6 r:0.2
       \move (1 12) 
       \fcir f:0 r:0.25
       \move (2 10) 
       \fcir f:0 r:0.25
       \move (1 10) 
       \fcir f:0 r:0.25
       \move (2 8) 
       \fcir f:0 r:0.25
       \move (0 12) 
       \fcir f:0 r:0.25
       \move (0 11) 
       \fcir f:0 r:0.25
       \move (3 5) 
       \fcir f:0 r:0.25
       \move (3 3) 
       \fcir f:0 r:0.25
       \move (1 9) 
       \fcir f:0 r:0.25
       \move (1 7) 
       \fcir f:0 r:0.25
       \move (1 6) 
       \fcir f:0 r:0.25
       \move (2 6) 
       \fcir f:0 r:0.25
       \move (0 9) 
       \fcir f:0 r:0.25
       \move (1 4) 
       \fcir f:0 r:0.25
       \move (3 1) 
       \fcir f:0 r:0.25
       \move (0 7) 
       \fcir f:0 r:0.25
       \move (1 2) 
       \fcir f:0 r:0.25
       \move (2 5) 
       \fcir f:0 r:0.25
       \move (2 3) 
       \fcir f:0 r:0.25
       \move (2 2) 
       \fcir f:0 r:0.25
       \move (2 0) 
       \fcir f:0 r:0.25
       \move (3 0) 
       \fcir f:0 r:0.25
   \end{texdraw}
   \end{center}

\noindent
The tropical curve is:
\vspace*{0.5cm}
   \begin{center}

    \begin{texdraw}
       \drawdim cm  \relunitscale 0.7 \arrowheadtype t:V
       %\linewd 0.05 \lpatt (0.1 0.4)
       %\move (-4 0) \avec (9 0) \move (0 -4) \avec (0 9)
       \linewd 0.1  \lpatt (1 0)

       \setgray 0.6
       \relunitscale 0.70
       \move (1 0.5) \fcir f:0 r:0.28
       \move (1 0.5) \lvec (-0.5 0.5)
       \htext (0.25 0){$2$}
       \move (1 0.5) \lvec (8.5 0.5)
       \htext (4.75 0){$2$}
       \move (1 0.5) \lvec (0 0)
       \move (1 0.5) \rlvec (2.5 1.25)
       \move (-1 1) \fcir f:0 r:0.28
       \move (-1 1) \lvec (-0.5 0.5)
       \move (-1 1) \rlvec (-2.5 0)
       \move (-1 1) \rlvec (0 2.5)
       \move (-0.5 0.5) \fcir f:0 r:0.28
       \move (-0.5 0.5) \lvec (-1 0)
       \move (8.5 0.5) \fcir f:0 r:0.28
       \move (8.5 0.5) \lvec (6 0)
       \move (8.5 0.5) \rlvec (2.5 0.5)
       \move (8.5 0.5) \rlvec (2.5 0)
       \htext (9.91 0.5){$2$}
       \move (-1 0) \fcir f:0 r:0.28
       \move (-1 0) \lvec (0 0)
       \htext (-0.5 -0.5){$4$}
       \move (-1 0) \lvec (-6 -1)
       \move (0 0) \fcir f:0 r:0.28
       \move (0 0) \lvec (6 0)
       \htext (3 -0.5){$2$}
       \move (0 0) \lvec (0 -1)
       \move (-6 -1) \fcir f:0 r:0.28
       \move (-6 -1) \lvec (-8.5 -1.5)
       \move (-6 -1) \lvec (0 -1)
       \htext (-3 -1.5){$2$}
       \move (-6 -1) \rlvec (-2.5 0)
       \htext (-7.41 -1){$2$}
       \move (6 0) \fcir f:0 r:0.28
       \move (6 0) \lvec (1 -1)
       \move (6 0) \rlvec (2.5 0)
       \htext (7.41 0){$2$}
       \move (-8.5 -1.5) \fcir f:0 r:0.28
       \move (-8.5 -1.5) \lvec (-1 -1.5)
       \htext (-4.75 -2){$2$}
       \move (-8.5 -1.5) \rlvec (-2.5 -0.5)
       \move (-8.5 -1.5) \rlvec (-2.5 0)
       \htext (-9.91 -1.5){$2$}
       \move (0 -1) \fcir f:0 r:0.28
       \move (0 -1) \lvec (1 -1)
       \htext (0.5 -1.5){$4$}
       \move (0 -1) \lvec (-1 -1.5)
       \move (1 -1) \fcir f:0 r:0.28
       \move (1 -1) \lvec (0.5 -1.5)
       \move (-1 -1.5) \fcir f:0 r:0.28
       \move (-1 -1.5) \lvec (0.5 -1.5)
       \htext (-0.25 -2){$2$}
       \move (-1 -1.5) \rlvec (-2.5 -1.25)
       \move (0.5 -1.5) \fcir f:0 r:0.28
       \move (0.5 -1.5) \lvec (1 -2)
       \move (1 -2) \fcir f:0 r:0.28
       \move (1 -2) \rlvec (2.5 0)
       \move (1 -2) \rlvec (0 -2.5)

   %% HERE STARTS THE CODE FOR THE LATTICE
        \move (-9 -3) \fcir f:0.8 r:0.14
        \move (-9 -2) \fcir f:0.8 r:0.14
        \move (-9 -1) \fcir f:0.8 r:0.14
        \move (-9 0) \fcir f:0.8 r:0.14
        \move (-9 1) \fcir f:0.8 r:0.14
        \move (-9 2) \fcir f:0.8 r:0.14
        \move (-8 -3) \fcir f:0.8 r:0.14
        \move (-8 -2) \fcir f:0.8 r:0.14
        \move (-8 -1) \fcir f:0.8 r:0.14
        \move (-8 0) \fcir f:0.8 r:0.14
        \move (-8 1) \fcir f:0.8 r:0.14
        \move (-8 2) \fcir f:0.8 r:0.14
        \move (-7 -3) \fcir f:0.8 r:0.14
        \move (-7 -2) \fcir f:0.8 r:0.14
        \move (-7 -1) \fcir f:0.8 r:0.14
        \move (-7 0) \fcir f:0.8 r:0.14
        \move (-7 1) \fcir f:0.8 r:0.14
        \move (-7 2) \fcir f:0.8 r:0.14
        \move (-6 -3) \fcir f:0.8 r:0.14
        \move (-6 -2) \fcir f:0.8 r:0.14
        \move (-6 -1) \fcir f:0.8 r:0.14
        \move (-6 0) \fcir f:0.8 r:0.14
        \move (-6 1) \fcir f:0.8 r:0.14
        \move (-6 2) \fcir f:0.8 r:0.14
        \move (-5 -3) \fcir f:0.8 r:0.14
        \move (-5 -2) \fcir f:0.8 r:0.14
        \move (-5 -1) \fcir f:0.8 r:0.14
        \move (-5 0) \fcir f:0.8 r:0.14
        \move (-5 1) \fcir f:0.8 r:0.14
        \move (-5 2) \fcir f:0.8 r:0.14
        \move (-4 -3) \fcir f:0.8 r:0.14
        \move (-4 -2) \fcir f:0.8 r:0.14
        \move (-4 -1) \fcir f:0.8 r:0.14
        \move (-4 0) \fcir f:0.8 r:0.14
        \move (-4 1) \fcir f:0.8 r:0.14
        \move (-4 2) \fcir f:0.8 r:0.14
        \move (-3 -3) \fcir f:0.8 r:0.14
        \move (-3 -2) \fcir f:0.8 r:0.14
        \move (-3 -1) \fcir f:0.8 r:0.14
        \move (-3 0) \fcir f:0.8 r:0.14
        \move (-3 1) \fcir f:0.8 r:0.14
        \move (-3 2) \fcir f:0.8 r:0.14
        \move (-2 -3) \fcir f:0.8 r:0.14
        \move (-2 -2) \fcir f:0.8 r:0.14
        \move (-2 -1) \fcir f:0.8 r:0.14
        \move (-2 0) \fcir f:0.8 r:0.14
        \move (-2 1) \fcir f:0.8 r:0.14
        \move (-2 2) \fcir f:0.8 r:0.14
        \move (-1 -3) \fcir f:0.8 r:0.14
        \move (-1 -2) \fcir f:0.8 r:0.14
        \move (-1 -1) \fcir f:0.8 r:0.14
        \move (-1 0) \fcir f:0.8 r:0.14
        \move (-1 1) \fcir f:0.8 r:0.14
        \move (-1 2) \fcir f:0.8 r:0.14
        \move (0 -3) \fcir f:0.8 r:0.14
        \move (0 -2) \fcir f:0.8 r:0.14
        \move (0 -1) \fcir f:0.8 r:0.14
        \move (0 0) \fcir f:0.8 r:0.14
        \move (0 1) \fcir f:0.8 r:0.14
        \move (0 2) \fcir f:0.8 r:0.14
        \move (1 -3) \fcir f:0.8 r:0.14
        \move (1 -2) \fcir f:0.8 r:0.14
        \move (1 -1) \fcir f:0.8 r:0.14
        \move (1 0) \fcir f:0.8 r:0.14
        \move (1 1) \fcir f:0.8 r:0.14
        \move (1 2) \fcir f:0.8 r:0.14
        \move (2 -3) \fcir f:0.8 r:0.14
        \move (2 -2) \fcir f:0.8 r:0.14
        \move (2 -1) \fcir f:0.8 r:0.14
        \move (2 0) \fcir f:0.8 r:0.14
        \move (2 1) \fcir f:0.8 r:0.14
        \move (2 2) \fcir f:0.8 r:0.14
        \move (3 -3) \fcir f:0.8 r:0.14
        \move (3 -2) \fcir f:0.8 r:0.14
        \move (3 -1) \fcir f:0.8 r:0.14
        \move (3 0) \fcir f:0.8 r:0.14
        \move (3 1) \fcir f:0.8 r:0.14
        \move (3 2) \fcir f:0.8 r:0.14
        \move (4 -3) \fcir f:0.8 r:0.14
        \move (4 -2) \fcir f:0.8 r:0.14
        \move (4 -1) \fcir f:0.8 r:0.14
        \move (4 0) \fcir f:0.8 r:0.14
        \move (4 1) \fcir f:0.8 r:0.14
        \move (4 2) \fcir f:0.8 r:0.14
        \move (5 -3) \fcir f:0.8 r:0.14
        \move (5 -2) \fcir f:0.8 r:0.14
        \move (5 -1) \fcir f:0.8 r:0.14
        \move (5 0) \fcir f:0.8 r:0.14
        \move (5 1) \fcir f:0.8 r:0.14
        \move (5 2) \fcir f:0.8 r:0.14
        \move (6 -3) \fcir f:0.8 r:0.14
        \move (6 -2) \fcir f:0.8 r:0.14
        \move (6 -1) \fcir f:0.8 r:0.14
        \move (6 0) \fcir f:0.8 r:0.14
        \move (6 1) \fcir f:0.8 r:0.14
        \move (6 2) \fcir f:0.8 r:0.14
        \move (7 -3) \fcir f:0.8 r:0.14
        \move (7 -2) \fcir f:0.8 r:0.14
        \move (7 -1) \fcir f:0.8 r:0.14
        \move (7 0) \fcir f:0.8 r:0.14
        \move (7 1) \fcir f:0.8 r:0.14
        \move (7 2) \fcir f:0.8 r:0.14
        \move (8 -3) \fcir f:0.8 r:0.14
        \move (8 -2) \fcir f:0.8 r:0.14
        \move (8 -1) \fcir f:0.8 r:0.14
        \move (8 0) \fcir f:0.8 r:0.14
        \move (8 1) \fcir f:0.8 r:0.14
        \move (8 2) \fcir f:0.8 r:0.14
        \move (9 -3) \fcir f:0.8 r:0.14
        \move (9 -2) \fcir f:0.8 r:0.14
        \move (9 -1) \fcir f:0.8 r:0.14
        \move (9 0) \fcir f:0.8 r:0.14
        \move (9 1) \fcir f:0.8 r:0.14
        \move (9 2) \fcir f:0.8 r:0.14
   %% HERE ENDS THE CODE FOR THE LATTICE
                          
    \end{texdraw}\end{center}

\subsection{The non-homogeneous tropical curve of the $6_1$ knot}
\lbl{sub.61non}

The non-homogeneous non-commutative $A$-polynomial $A^{nh}_{6_1}(M,L,q)$ has
$346$ terms, and it is given by \cite{GS} (with the notation $A^{nh}_{-2}(E,Q,q)$
where $E=L$, $Q=M$):

\begin{center} 
     \begin{math}
A^{nh}_{6_1}(y,x,1/t)=\tfrac{1}{t^{31}}\cdot x^{2}y^{15}+\tfrac{-1-t}{t^{28}}\cdot x^{3}y^{13}+\tfrac{-1-3\cdot t}{t^{30}}\cdot x^{2}y^{14}+\tfrac{-1}{t^{30}}\cdot xy^{15}+\tfrac{1}{t^{22}}\cdot x^{4}y^{11}+\tfrac{1+3\cdot t+t^{2}}{t^{26}}\cdot x^{3}y^{12}+\tfrac{-1-t+2\cdot t^{2}+2\cdot t^{3}-t^{4}-t^{5}-t^{6}}{t^{30}}\cdot x^{2}y^{13}+\tfrac{2}{t^{29}}\cdot xy^{14}+\tfrac{-1}{t^{18}}\cdot x^{4}y^{10}+\tfrac{1+2\cdot t+2\cdot t^{2}+t^{3}-t^{4}-2\cdot t^{5}+2\cdot t^{6}+t^{7}}{t^{27}}\cdot x^{3}y^{11}+\tfrac{1+4\cdot t+3\cdot t^{2}-t^{3}+2\cdot t^{5}+4\cdot t^{6}+3\cdot t^{7}}{t^{29}}\cdot x^{2}y^{12}+\tfrac{1-t-t^{2}+t^{4}+t^{5}+t^{6}}{t^{29}}\cdot xy^{13}+\tfrac{-1-t-t^{2}}{t^{21}}\cdot x^{4}y^{9}+\tfrac{-1-4\cdot t-4\cdot t^{2}-3\cdot t^{3}-2\cdot t^{6}-3\cdot t^{7}-t^{8}}{t^{25}}\cdot x^{3}y^{10}+\tfrac{1-2\cdot t-3\cdot t^{2}+3\cdot t^{4}+3\cdot t^{5}+t^{6}-4\cdot t^{7}-t^{8}+t^{9}+t^{10}}{t^{28}}\cdot x^{2}y^{11}+\tfrac{-2+2\cdot t^{2}-2\cdot t^{4}-2\cdot t^{5}-2\cdot t^{6}}{t^{28}}\cdot xy^{12}+\tfrac{1+t+t^{2}}{t^{17}}\cdot x^{4}y^{8}+\tfrac{-1-2\cdot t-2\cdot t^{2}+3\cdot t^{4}-2\cdot t^{6}-4\cdot t^{7}-t^{8}+3\cdot t^{9}+t^{10}-t^{11}}{t^{25}}\cdot x^{3}y^{9}+\tfrac{-1-3\cdot t+t^{2}-3\cdot t^{4}-6\cdot t^{5}-6\cdot t^{6}-4\cdot t^{7}+t^{8}-t^{9}-2\cdot t^{10}-3\cdot t^{11}}{t^{27}}\cdot x^{2}y^{10}+\tfrac{1+2\cdot t-t^{3}+t^{5}+2\cdot t^{6}+t^{7}-t^{8}-t^{9}-t^{10}}{t^{27}}\cdot xy^{11}+\tfrac{1+t+t^{2}}{t^{19}}\cdot x^{4}y^{7}+\tfrac{1+3\cdot t+3\cdot t^{2}+2\cdot t^{3}+2\cdot t^{5}+5\cdot t^{6}+5\cdot t^{7}+3\cdot t^{8}-t^{10}+t^{11}+t^{12}}{t^{23}}\cdot x^{3}y^{8}+\tfrac{1+t-2\cdot t^{2}-2\cdot t^{3}-2\cdot t^{4}+2\cdot t^{5}-t^{7}-3\cdot t^{8}-3\cdot t^{9}+2\cdot t^{11}-t^{12}}{t^{25}}\cdot x^{2}y^{9}+\tfrac{1-3\cdot t-t^{2}+3\cdot t^{3}+3\cdot t^{4}+t^{5}-2\cdot t^{6}-2\cdot t^{7}+2\cdot t^{8}+2\cdot t^{9}+2\cdot t^{10}}{t^{26}}\cdot xy^{10}+\tfrac{1}{t^{26}}\cdot y^{11}+\tfrac{-1-t-t^{2}}{t^{15}}\cdot x^{4}y^{6}+\tfrac{1+t-t^{2}-2\cdot t^{3}+t^{4}+3\cdot t^{5}+4\cdot t^{6}-2\cdot t^{7}-4\cdot t^{8}-2\cdot t^{9}+t^{10}+2\cdot t^{11}-2\cdot t^{13}}{t^{22}}\cdot x^{3}y^{7}+\tfrac{1+2\cdot t+2\cdot t^{2}+3\cdot t^{3}+t^{4}+3\cdot t^{5}+3\cdot t^{6}+3\cdot t^{7}+3\cdot t^{8}+t^{9}+t^{11}}{t^{22}}\cdot x^{2}y^{8}+\tfrac{-2-2\cdot t^{3}-4\cdot t^{4}-4\cdot t^{5}-2\cdot t^{6}+t^{7}-t^{9}-2\cdot t^{10}-t^{11}+t^{13}}{t^{25}}\cdot xy^{9}+\tfrac{-1}{t^{26}}\cdot y^{10}+\tfrac{-1}{t^{16}}\cdot x^{4}y^{5}+\tfrac{-2-2\cdot t^{3}-4\cdot t^{4}-4\cdot t^{5}-2\cdot t^{6}+t^{7}-t^{9}-2\cdot t^{10}-t^{11}+t^{13}}{t^{19}}\cdot x^{3}y^{6}+\tfrac{1+2\cdot t+2\cdot t^{2}+3\cdot t^{3}+t^{4}+3\cdot t^{5}+3\cdot t^{6}+3\cdot t^{7}+3\cdot t^{8}+t^{9}+t^{11}}{t^{20}}\cdot x^{2}y^{7}+\tfrac{1+t-t^{2}-2\cdot t^{3}+t^{4}+3\cdot t^{5}+4\cdot t^{6}-2\cdot t^{7}-4\cdot t^{8}-2\cdot t^{9}+t^{10}+2\cdot t^{11}-2\cdot t^{13}}{t^{24}}\cdot xy^{8}+\tfrac{-1-t-t^{2}}{t^{21}}\cdot y^{9}+\tfrac{1}{t^{12}}\cdot x^{4}y^{4}+\tfrac{1-3\cdot t-t^{2}+3\cdot t^{3}+3\cdot t^{4}+t^{5}-2\cdot t^{6}-2\cdot t^{7}+2\cdot t^{8}+2\cdot t^{9}+2\cdot t^{10}}{t^{16}}\cdot x^{3}y^{5}+\tfrac{1+t-2\cdot t^{2}-2\cdot t^{3}-2\cdot t^{4}+2\cdot t^{5}-t^{7}-3\cdot t^{8}-3\cdot t^{9}+2\cdot t^{11}-t^{12}}{t^{19}}\cdot x^{2}y^{6}+\tfrac{1+3\cdot t+3\cdot t^{2}+2\cdot t^{3}+2\cdot t^{5}+5\cdot t^{6}+5\cdot t^{7}+3\cdot t^{8}-t^{10}+t^{11}+t^{12}}{t^{21}}\cdot xy^{7}+\tfrac{1+t+t^{2}}{t^{21}}\cdot y^{8}+\tfrac{1+2\cdot t-t^{3}+t^{5}+2\cdot t^{6}+t^{7}-t^{8}-t^{9}-t^{10}}{t^{13}}\cdot x^{3}y^{4}+\tfrac{-1-3\cdot t+t^{2}-3\cdot t^{4}-6\cdot t^{5}-6\cdot t^{6}-4\cdot t^{7}+t^{8}-t^{9}-2\cdot t^{10}-3\cdot t^{11}}{t^{17}}\cdot x^{2}y^{5}+\tfrac{-1-2\cdot t-2\cdot t^{2}+3\cdot t^{4}-2\cdot t^{6}-4\cdot t^{7}-t^{8}+3\cdot t^{9}+t^{10}-t^{11}}{t^{19}}\cdot xy^{6}+\tfrac{1+t+t^{2}}{t^{15}}\cdot y^{7}+\tfrac{-2+2\cdot t^{2}-2\cdot t^{4}-2\cdot t^{5}-2\cdot t^{6}}{t^{10}}\cdot x^{3}y^{3}+\tfrac{1-2\cdot t-3\cdot t^{2}+3\cdot t^{4}+3\cdot t^{5}+t^{6}-4\cdot t^{7}-t^{8}+t^{9}+t^{10}}{t^{14}}\cdot x^{2}y^{4}+\tfrac{-1-4\cdot t-4\cdot t^{2}-3\cdot t^{3}-2\cdot t^{6}-3\cdot t^{7}-t^{8}}{t^{15}}\cdot xy^{5}+\tfrac{-1-t-t^{2}}{t^{15}}\cdot y^{6}+\tfrac{1-t-t^{2}+t^{4}+t^{5}+t^{6}}{t^{7}}\cdot x^{3}y^{2}+\tfrac{1+4\cdot t+3\cdot t^{2}-t^{3}+2\cdot t^{5}+4\cdot t^{6}+3\cdot t^{7}}{t^{11}}\cdot x^{2}y^{3}+\tfrac{1+2\cdot t+2\cdot t^{2}+t^{3}-t^{4}-2\cdot t^{5}+2\cdot t^{6}+t^{7}}{t^{13}}\cdot xy^{4}+\tfrac{-1}{t^{8}}\cdot y^{5}+\tfrac{2}{t^{3}}\cdot x^{3}y+\tfrac{-1-t+2\cdot t^{2}+2\cdot t^{3}-t^{4}-t^{5}-t^{6}}{t^{8}}\cdot x^{2}y^{2}+\tfrac{1+3\cdot t+t^{2}}{t^{8}}\cdot xy^{3}+\tfrac{1}{t^{8}}\cdot y^{4}-x^{3}+\tfrac{-1-3\cdot t}{t^{4}}\cdot x^{2}y+\tfrac{-1-t}{t^{6}}\cdot xy^{2}+\tfrac{1}{t}\cdot x^{2}
      \end{math}
   \end{center}
   \vspace*{0.5cm}
   The vertices of the tropical curve are:
   \begin{displaymath}
      (2,-1/2),\;\;(-1,-1/2),\;\;(5,-1/2),\;\;(-3/2,-1/2),\;\;(-4,0),\;\;(1,-1),\;\;(-1/2,-1),\;\;
    \end{displaymath}
    \begin{displaymath}
(-1,-2/3),\;\;(4,-1),\;\;(1/2,-3/2),\;\;(3,-3/2),\;\;(1/5,-8/5),\;\;(-1/2,-5/4),\;\;
    \end{displaymath}
    \begin{displaymath}
(1/2,-11/4),\;\;(-1/5,-12/5),\;\;(-3,-5/2),\;\;(4,-4),\;\;(1/2,-3),\;\;(1,-10/3),\;\;(3/2,-7/2),\;\;
   \end{displaymath}
    \begin{displaymath}
(-1/2,-5/2),\;\;(-4,-3),\;\;(-1,-3),\;\;(-5,-7/2),\;\;(1,-7/2),\;\;(-2,-7/2)
\end{displaymath}
   \vspace*{0.5cm}
The tropical curve is:
   
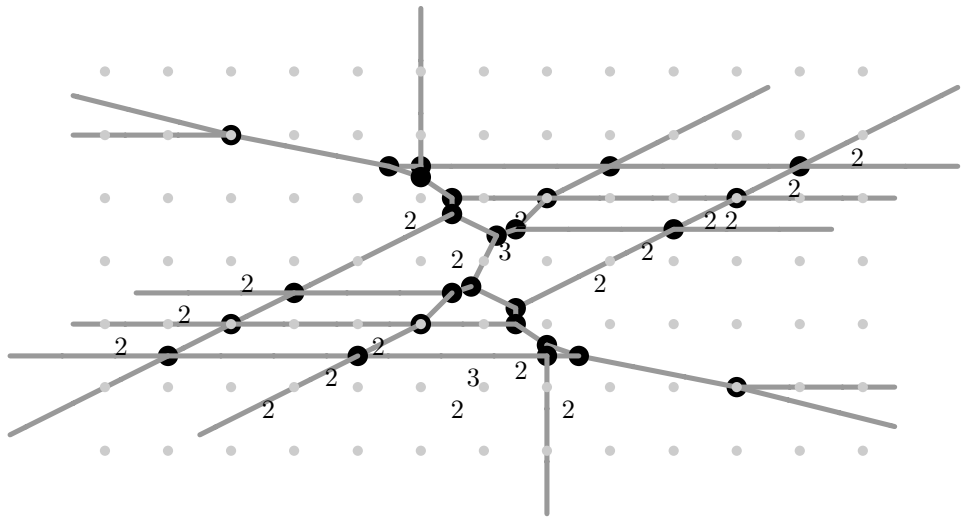
\begin{center}

    \begin{texdraw}
       \drawdim cm  \relunitscale 0.7 \arrowheadtype t:V
       %\linewd 0.05 \lpatt (0.1 0.4)
       %\move (-4 0) \avec (9 0) \move (0 -4) \avec (0 9)
       \linewd 0.1  \lpatt (1 0)

       \setgray 0.6
       \relunitscale 1.2
       \move (2 1.5) \fcir f:0 r:0.16
       \move (2 1.5) \lvec (-1 1.5)
       \htext (0.5 0.5){$2$}
       \move (2 1.5) \lvec (5 1.5)
       \htext (3.5 0.5){$2$}
       \move (2 1.5) \lvec (1 1)
       \move (2 1.5) \rlvec (2.5 1.25)
       \move (-1 1.5) \fcir f:0 r:0.16
       \move (-1 1.5) \lvec (-1.5 1.5)
       \htext (-1.25 0.5){$2$}
       \move (-1 1.5) \lvec (-1 1.33)
       \move (-1 1.5) \rlvec (0 2.5)
       \move (5 1.5) \fcir f:0 r:0.16
       \move (5 1.5) \lvec (4 1)
       \move (5 1.5) \rlvec (2.5 1.25)
       \move (5 1.5) \rlvec (2.5 0)
       \htext (5.83 1.5){$2$}
       \move (-1.5 1.5) \fcir f:0 r:0.16
       \move (-1.5 1.5) \lvec (-4 2)
       \move (-1.5 1.5) \lvec (-1 1.33)
       \move (-4 2) \fcir f:0 r:0.16
       \move (-4 2) \rlvec (-2.5 0)
       \move (-4 2) \rlvec (-2.5 0.625)
       \move (1 1) \fcir f:0 r:0.16
       \move (1 1) \lvec (-0.5 1)
       \htext (0.25 0){$3$}
       \move (1 1) \lvec (4 1)
       \htext (2.5 0){$2$}
       \move (1 1) \lvec (0.5 0.5)
       \move (-0.5 1) \fcir f:0 r:0.16
       \move (-0.5 1) \lvec (-1 1.33)
       \move (-0.5 1) \lvec (-0.5 0.75)
       \htext (-0.5 -0.12){$2$}
       \move (-1 1.33) \fcir f:0 r:0.16
       \move (4 1) \fcir f:0 r:0.16
       \move (4 1) \lvec (3 0.5)
       \move (4 1) \rlvec (2.5 0)
       \htext (4.83 1){$2$}
       \move (0.5 0.5) \fcir f:0 r:0.16
       \move (0.5 0.5) \lvec (3 0.5)
       \htext (1.75 -0.5){$2$}
       \move (0.5 0.5) \lvec (0.2 0.4)
       \move (3 0.5) \fcir f:0 r:0.16
       \move (3 0.5) \lvec (0.5 -0.75)
       \move (3 0.5) \rlvec (2.5 0)
       \htext (3.83 0.5){$2$}
       \move (0.2 0.4) \fcir f:0 r:0.16
       \move (0.2 0.4) \lvec (-0.5 0.75)
       \move (0.2 0.4) \lvec (-0.2 -0.4)
       \move (-0.5 0.75) \fcir f:0 r:0.16
       \move (-0.5 0.75) \lvec (-3 -0.5)
       \move (0.5 -0.75) \fcir f:0 r:0.16
       \move (0.5 -0.75) \lvec (-0.2 -0.4)
       \move (0.5 -0.75) \lvec (0.5 -1)
       \htext (0.5 -1.87){$2$}
       \move (-0.2 -0.4) \fcir f:0 r:0.16
       \move (-0.2 -0.4) \lvec (-0.5 -0.5)
       \move (-3 -0.5) \fcir f:0 r:0.16
       \move (-3 -0.5) \lvec (-0.5 -0.5)
       \htext (-1.75 -1.5){$2$}
       \move (-3 -0.5) \lvec (-4 -1)
       \move (-3 -0.5) \rlvec (-2.5 0)
       \htext (-3.83 -0.5){$2$}
       \move (4 -2) \fcir f:0 r:0.16
       \move (4 -2) \lvec (1.5 -1.5)
       \move (4 -2) \rlvec (2.5 0)
       \move (4 -2) \rlvec (2.5 -0.625)
       \move (0.5 -1) \fcir f:0 r:0.16
       \move (0.5 -1) \lvec (1 -1.33)
       \move (0.5 -1) \lvec (-1 -1)
       \htext (-0.25 -2){$3$}
       \move (1 -1.33) \fcir f:0 r:0.16
       \move (1 -1.33) \lvec (1.5 -1.5)
       \move (1 -1.33) \lvec (1 -1.5)
       \move (1.5 -1.5) \fcir f:0 r:0.16
       \move (1.5 -1.5) \lvec (1 -1.5)
       \htext (1.25 -2.5){$2$}
       \move (-0.5 -0.5) \fcir f:0 r:0.16
       \move (-0.5 -0.5) \lvec (-1 -1)
       \move (-4 -1) \fcir f:0 r:0.16
       \move (-4 -1) \lvec (-1 -1)
       \htext (-2.5 -2){$2$}
       \move (-4 -1) \lvec (-5 -1.5)
       \move (-4 -1) \rlvec (-2.5 0)
       \htext (-4.83 -1){$2$}
       \move (-1 -1) \fcir f:0 r:0.16
       \move (-1 -1) \lvec (-2 -1.5)
       \move (-5 -1.5) \fcir f:0 r:0.16
       \move (-5 -1.5) \lvec (-2 -1.5)
       \htext (-3.5 -2.5){$2$}
       \move (-5 -1.5) \rlvec (-2.5 -1.25)
       \move (-5 -1.5) \rlvec (-2.5 0)
       \htext (-5.83 -1.5){$2$}
       \move (1 -1.5) \fcir f:0 r:0.16
       \move (1 -1.5) \lvec (-2 -1.5)
       \htext (-0.5 -2.5){$2$}
       \move (1 -1.5) \rlvec (0 -2.5)
       \move (-2 -1.5) \fcir f:0 r:0.16
       \move (-2 -1.5) \rlvec (-2.5 -1.25)

   %% HERE STARTS THE CODE FOR THE LATTICE
        \move (-6 -3) \fcir f:0.8 r:0.08
        \move (-6 -2) \fcir f:0.8 r:0.08
        \move (-6 -1) \fcir f:0.8 r:0.08
        \move (-6 0) \fcir f:0.8 r:0.08
        \move (-6 1) \fcir f:0.8 r:0.08
        \move (-6 2) \fcir f:0.8 r:0.08
        \move (-6 3) \fcir f:0.8 r:0.08
        \move (-5 -3) \fcir f:0.8 r:0.08
        \move (-5 -2) \fcir f:0.8 r:0.08
        \move (-5 -1) \fcir f:0.8 r:0.08
        \move (-5 0) \fcir f:0.8 r:0.08
        \move (-5 1) \fcir f:0.8 r:0.08
        \move (-5 2) \fcir f:0.8 r:0.08
        \move (-5 3) \fcir f:0.8 r:0.08
        \move (-4 -3) \fcir f:0.8 r:0.08
        \move (-4 -2) \fcir f:0.8 r:0.08
        \move (-4 -1) \fcir f:0.8 r:0.08
        \move (-4 0) \fcir f:0.8 r:0.08
        \move (-4 1) \fcir f:0.8 r:0.08
        \move (-4 2) \fcir f:0.8 r:0.08
        \move (-4 3) \fcir f:0.8 r:0.08
        \move (-3 -3) \fcir f:0.8 r:0.08
        \move (-3 -2) \fcir f:0.8 r:0.08
        \move (-3 -1) \fcir f:0.8 r:0.08
        \move (-3 0) \fcir f:0.8 r:0.08
        \move (-3 1) \fcir f:0.8 r:0.08
        \move (-3 2) \fcir f:0.8 r:0.08
        \move (-3 3) \fcir f:0.8 r:0.08
        \move (-2 -3) \fcir f:0.8 r:0.08
        \move (-2 -2) \fcir f:0.8 r:0.08
        \move (-2 -1) \fcir f:0.8 r:0.08
        \move (-2 0) \fcir f:0.8 r:0.08
        \move (-2 1) \fcir f:0.8 r:0.08
        \move (-2 2) \fcir f:0.8 r:0.08
        \move (-2 3) \fcir f:0.8 r:0.08
        \move (-1 -3) \fcir f:0.8 r:0.08
        \move (-1 -2) \fcir f:0.8 r:0.08
        \move (-1 -1) \fcir f:0.8 r:0.08
        \move (-1 0) \fcir f:0.8 r:0.08
        \move (-1 1) \fcir f:0.8 r:0.08
        \move (-1 2) \fcir f:0.8 r:0.08
        \move (-1 3) \fcir f:0.8 r:0.08
        \move (0 -3) \fcir f:0.8 r:0.08
        \move (0 -2) \fcir f:0.8 r:0.08
        \move (0 -1) \fcir f:0.8 r:0.08
        \move (0 0) \fcir f:0.8 r:0.08
        \move (0 1) \fcir f:0.8 r:0.08
        \move (0 2) \fcir f:0.8 r:0.08
        \move (0 3) \fcir f:0.8 r:0.08
        \move (1 -3) \fcir f:0.8 r:0.08
        \move (1 -2) \fcir f:0.8 r:0.08
        \move (1 -1) \fcir f:0.8 r:0.08
        \move (1 0) \fcir f:0.8 r:0.08
        \move (1 1) \fcir f:0.8 r:0.08
        \move (1 2) \fcir f:0.8 r:0.08
        \move (1 3) \fcir f:0.8 r:0.08
        \move (2 -3) \fcir f:0.8 r:0.08
        \move (2 -2) \fcir f:0.8 r:0.08
        \move (2 -1) \fcir f:0.8 r:0.08
        \move (2 0) \fcir f:0.8 r:0.08
        \move (2 1) \fcir f:0.8 r:0.08
        \move (2 2) \fcir f:0.8 r:0.08
        \move (2 3) \fcir f:0.8 r:0.08
        \move (3 -3) \fcir f:0.8 r:0.08
        \move (3 -2) \fcir f:0.8 r:0.08
        \move (3 -1) \fcir f:0.8 r:0.08
        \move (3 0) \fcir f:0.8 r:0.08
        \move (3 1) \fcir f:0.8 r:0.08
        \move (3 2) \fcir f:0.8 r:0.08
        \move (3 3) \fcir f:0.8 r:0.08
        \move (4 -3) \fcir f:0.8 r:0.08
        \move (4 -2) \fcir f:0.8 r:0.08
        \move (4 -1) \fcir f:0.8 r:0.08
        \move (4 0) \fcir f:0.8 r:0.08
        \move (4 1) \fcir f:0.8 r:0.08
        \move (4 2) \fcir f:0.8 r:0.08
        \move (4 3) \fcir f:0.8 r:0.08
        \move (5 -3) \fcir f:0.8 r:0.08
        \move (5 -2) \fcir f:0.8 r:0.08
        \move (5 -1) \fcir f:0.8 r:0.08
        \move (5 0) \fcir f:0.8 r:0.08
        \move (5 1) \fcir f:0.8 r:0.08
        \move (5 2) \fcir f:0.8 r:0.08
        \move (5 3) \fcir f:0.8 r:0.08
        \move (6 -3) \fcir f:0.8 r:0.08
        \move (6 -2) \fcir f:0.8 r:0.08
        \move (6 -1) \fcir f:0.8 r:0.08
        \move (6 0) \fcir f:0.8 r:0.08
        \move (6 1) \fcir f:0.8 r:0.08
        \move (6 2) \fcir f:0.8 r:0.08
        \move (6 3) \fcir f:0.8 r:0.08
   %% HERE ENDS THE CODE FOR THE LATTICE
                          
    \end{texdraw}\end{center}

   \vspace*{0.5cm}
   The Newton subdivision of the tropical curve is:
   
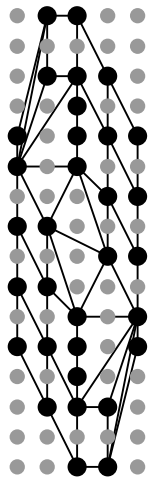
\vspace*{0.5cm}

   \begin{center}
            
    \begin{texdraw}
       \drawdim cm  \relunitscale 0.5  %% originally: 1 
       \linewd 0.05
       \relunitscale0.8
        \move (2 15)        
        \lvec (4 11)
        \move (4 11)        
        \lvec (4 4)
        \move (4 4)        
        \lvec (3 0)
        \move (3 0)        
        \lvec (2 0)
        \move (2 0)        
        \lvec (0 4)
        \move (0 4)        
        \lvec (0 11)
        \move (0 11)        
        \lvec (1 15)
        \move (1 15)        
        \lvec (2 15)

        \move (2 13)        
        \lvec (2 15)
        \move (3 13)        
        \lvec (3 11)
        \move (3 11)        
        \lvec (2 13)
        \move (1 13)        
        \lvec (1 15)
        \move (2 13)        
        \lvec (1 13)
        \move (4 9)        
        \lvec (3 11)
        \move (0 10)        
        \lvec (1 15)
        \move (1 13)        
        \lvec (0 10)
        \move (2 12)        
        \lvec (2 13)
        \move (2 11)        
        \lvec (2 12)
        \move (2 10)        
        \lvec (2 11)
        \move (3 11)        
        \lvec (3 9)
        \move (3 9)        
        \lvec (2 10)
        \move (0 10)        
        \lvec (2 13)
        \move (2 10)        
        \lvec (0 10)
        \move (4 7)        
        \lvec (3 9)
        \move (3 9)        
        \lvec (3 7)
        \move (3 7)        
        \lvec (2 10)
        \move (4 5)        
        \lvec (3 7)
        \move (1 8)        
        \lvec (2 10)
        \move (3 7)        
        \lvec (1 8)
        \move (1 8)        
        \lvec (0 10)
        \move (2 5)        
        \lvec (3 7)
        \move (4 5)        
        \lvec (2 5)
        \move (2 5)        
        \lvec (1 8)
        \move (1 8)        
        \lvec (1 6)
        \move (1 6)        
        \lvec (0 8)
        \move (3 0)        
        \lvec (4 5)
        \move (4 5)        
        \lvec (2 2)
        \move (2 4)        
        \lvec (2 5)
        \move (2 3)        
        \lvec (2 4)
        \move (2 2)        
        \lvec (2 3)
        \move (4 5)        
        \lvec (3 2)
        \move (3 2)        
        \lvec (2 2)
        \move (3 0)        
        \lvec (3 2)
        \move (2 5)        
        \lvec (1 6)
        \move (1 6)        
        \lvec (1 4)
        \move (1 4)        
        \lvec (0 6)
        \move (2 2)        
        \lvec (1 4)
        \move (1 4)        
        \lvec (1 2)
        \move (2 0)        
        \lvec (2 2)
        \move (0 0) \fcir f:0.6 r:0.25
        \move (0 1) \fcir f:0.6 r:0.25
        \move (0 2) \fcir f:0.6 r:0.25
        \move (0 3) \fcir f:0.6 r:0.25
        \move (0 4) \fcir f:0.6 r:0.25
        \move (0 5) \fcir f:0.6 r:0.25
        \move (0 6) \fcir f:0.6 r:0.25
        \move (0 7) \fcir f:0.6 r:0.25
        \move (0 8) \fcir f:0.6 r:0.25
        \move (0 9) \fcir f:0.6 r:0.25
        \move (0 10) \fcir f:0.6 r:0.25
        \move (0 11) \fcir f:0.6 r:0.25
        \move (0 12) \fcir f:0.6 r:0.25
        \move (0 13) \fcir f:0.6 r:0.25
        \move (0 14) \fcir f:0.6 r:0.25
        \move (0 15) \fcir f:0.6 r:0.25
        \move (1 0) \fcir f:0.6 r:0.25
        \move (1 1) \fcir f:0.6 r:0.25
        \move (1 2) \fcir f:0.6 r:0.25
        \move (1 3) \fcir f:0.6 r:0.25
        \move (1 4) \fcir f:0.6 r:0.25
        \move (1 5) \fcir f:0.6 r:0.25
        \move (1 6) \fcir f:0.6 r:0.25
        \move (1 7) \fcir f:0.6 r:0.25
        \move (1 8) \fcir f:0.6 r:0.25
        \move (1 9) \fcir f:0.6 r:0.25
        \move (1 10) \fcir f:0.6 r:0.25
        \move (1 11) \fcir f:0.6 r:0.25
        \move (1 12) \fcir f:0.6 r:0.25
        \move (1 13) \fcir f:0.6 r:0.25
        \move (1 14) \fcir f:0.6 r:0.25
        \move (1 15) \fcir f:0.6 r:0.25
        \move (2 0) \fcir f:0.6 r:0.25
        \move (2 1) \fcir f:0.6 r:0.25
        \move (2 2) \fcir f:0.6 r:0.25
        \move (2 3) \fcir f:0.6 r:0.25
        \move (2 4) \fcir f:0.6 r:0.25
        \move (2 5) \fcir f:0.6 r:0.25
        \move (2 6) \fcir f:0.6 r:0.25
        \move (2 7) \fcir f:0.6 r:0.25
        \move (2 8) \fcir f:0.6 r:0.25
        \move (2 9) \fcir f:0.6 r:0.25
        \move (2 10) \fcir f:0.6 r:0.25
        \move (2 11) \fcir f:0.6 r:0.25
        \move (2 12) \fcir f:0.6 r:0.25
        \move (2 13) \fcir f:0.6 r:0.25
        \move (2 14) \fcir f:0.6 r:0.25
        \move (2 15) \fcir f:0.6 r:0.25
        \move (3 0) \fcir f:0.6 r:0.25
        \move (3 1) \fcir f:0.6 r:0.25
        \move (3 2) \fcir f:0.6 r:0.25
        \move (3 3) \fcir f:0.6 r:0.25
        \move (3 4) \fcir f:0.6 r:0.25
        \move (3 5) \fcir f:0.6 r:0.25
        \move (3 6) \fcir f:0.6 r:0.25
        \move (3 7) \fcir f:0.6 r:0.25
        \move (3 8) \fcir f:0.6 r:0.25
        \move (3 9) \fcir f:0.6 r:0.25
        \move (3 10) \fcir f:0.6 r:0.25
        \move (3 11) \fcir f:0.6 r:0.25
        \move (3 12) \fcir f:0.6 r:0.25
        \move (3 13) \fcir f:0.6 r:0.25
        \move (3 14) \fcir f:0.6 r:0.25
        \move (3 15) \fcir f:0.6 r:0.25
        \move (4 0) \fcir f:0.6 r:0.25
        \move (4 1) \fcir f:0.6 r:0.25
        \move (4 2) \fcir f:0.6 r:0.25
        \move (4 3) \fcir f:0.6 r:0.25
        \move (4 4) \fcir f:0.6 r:0.25
        \move (4 5) \fcir f:0.6 r:0.25
        \move (4 6) \fcir f:0.6 r:0.25
        \move (4 7) \fcir f:0.6 r:0.25
        \move (4 8) \fcir f:0.6 r:0.25
        \move (4 9) \fcir f:0.6 r:0.25
        \move (4 10) \fcir f:0.6 r:0.25
        \move (4 11) \fcir f:0.6 r:0.25
        \move (4 12) \fcir f:0.6 r:0.25
        \move (4 13) \fcir f:0.6 r:0.25
        \move (4 14) \fcir f:0.6 r:0.25
        \move (4 15) \fcir f:0.6 r:0.25
       \move (2 15) 
       \fcir f:0 r:0.31
       \move (3 13) 
       \fcir f:0 r:0.31
       \move (2 13) 
       \fcir f:0 r:0.31
       \move (3 11) 
       \fcir f:0 r:0.31
       \move (1 15) 
       \fcir f:0 r:0.31
       \move (1 13) 
       \fcir f:0 r:0.31
       \move (4 11) 
       \fcir f:0 r:0.31
       \move (4 9) 
       \fcir f:0 r:0.31
       \move (0 10) 
       \fcir f:0 r:0.31
       \move (0 11) 
       \fcir f:0 r:0.31
       \move (2 12) 
       \fcir f:0 r:0.31
       \move (2 11) 
       \fcir f:0 r:0.31
       \move (3 9) 
       \fcir f:0 r:0.31
       \move (2 10) 
       \fcir f:0 r:0.31
       \move (4 7) 
       \fcir f:0 r:0.31
       \move (3 7) 
       \fcir f:0 r:0.31
       \move (4 5) 
       \fcir f:0 r:0.31
       \move (1 8) 
       \fcir f:0 r:0.31
       \move (2 5) 
       \fcir f:0 r:0.31
       \move (0 8) 
       \fcir f:0 r:0.31
       \move (1 6) 
       \fcir f:0 r:0.31
       \move (4 4) 
       \fcir f:0 r:0.31
       \move (3 0) 
       \fcir f:0 r:0.31
       \move (2 4) 
       \fcir f:0 r:0.31
       \move (2 3) 
       \fcir f:0 r:0.31
       \move (2 2) 
       \fcir f:0 r:0.31
       \move (3 2) 
       \fcir f:0 r:0.31
       \move (0 6) 
       \fcir f:0 r:0.31
       \move (1 4) 
       \fcir f:0 r:0.31
       \move (0 4) 
       \fcir f:0 r:0.31
       \move (1 2) 
       \fcir f:0 r:0.31
       \move (2 0) 
       \fcir f:0 r:0.31
   \end{texdraw}
        \end{center}

\subsection{The non-homogeneous tropical curve of the $8_1$ knot}
\lbl{sub.81non}

The non-homogeneous non-commutative $A$-polynomial $A^{nh}_{8_1}(M,L,q)$ has
$2112$ terms, which we not present here. The vertices of the tropical curve 
are:

   \begin{displaymath}
      (3,-1/2),\;\;(-1,-1/2),\;\;(6,-1/2),\;\;(-2,-1/2),\;\;(9,-1/2),\;\;(2,-1),\;\;(-1,-1),\;\;(-5/2,-1/2),\;\;
\end{displaymath}
\begin{displaymath}
(-6,0),\;\;(5,-1),\;\;(-2,-3/5),\;\;(8,-1),\;\;(3/2,-3/2),\;\;(4,-3/2),\;\;(-1/2,-3/2),\;\;(-3/4,-11/8),\;\;
\end{displaymath}
\begin{displaymath}
(7,-3/2),\;\;(1,-2),\;\;(3,-2),\;\;(0,-2),\;\;(6,-2),\;\;(0,-5/2),\;\;(5/2,-5/2),\;\;(5,-5/2),\;\;(1,-3),\;\;(0,-3),\;\;
\end{displaymath}
\begin{displaymath}
(-5,-7/2),\;\;(0,-7/2),\;\;(-1,-3),\;\;(-5/2,-7/2),\;\;(3/4,-37/8),\;\;(0,-4),\;\;(1/2,-9/2),\;\;(-6,-4),\;\;
\end{displaymath}
\begin{displaymath}
(6,-6),\;\;(1,-5),\;\;(2,-27/5),\;\;(5/2,-11/2),\;\;(-3,-4),\;\;(-1,-4),\;\;(-7,-9/2),\;\;(-4,-9/2),\;\;
\end{displaymath}
\begin{displaymath}
(-3/2,-9/2),\;\;(-8,-5),\;\;(-5,-5),\;\;(-2,-5),\;\;(-9,-11/2),\;\;(2,-11/2),\;\;(-6,-11/2),\;\;(1,-11/2),\;\;
\end{displaymath}
\begin{displaymath}
(-3,-11/2)
   \end{displaymath}
   \vspace*{0.5cm}

The tropical curve is:

   
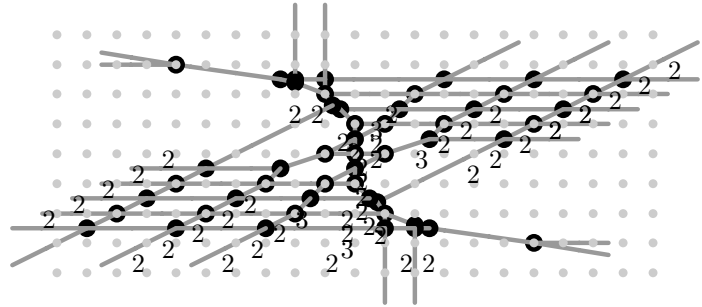
\begin{center}

    \begin{texdraw}
       \drawdim cm  \relunitscale 0.6 \arrowheadtype t:V
       %\drawdim cm  \relunitscale 0.7 \arrowheadtype t:V
       %\linewd 0.05 \lpatt (0.1 0.4)
       %\move (-4 0) \avec (9 0) \move (0 -4) \avec (0 9)
       \linewd 0.1  \lpatt (1 0)

       \setgray 0.6
       \relunitscale 0.66
       \move (3 2.5) \fcir f:0 r:0.3
       \move (3 2.5) \lvec (-1 2.5)
       \htext (1 1){$2$}
       \move (3 2.5) \lvec (6 2.5)
       \htext (4.5 1){$2$}
       \move (3 2.5) \lvec (2 2)
       \move (3 2.5) \rlvec (2.5 1.25)
       \move (-1 2.5) \fcir f:0 r:0.3
       \move (-1 2.5) \lvec (-2 2.5)
       \htext (-1.5 1){$2$}
       \move (-1 2.5) \lvec (-1 2)
       \move (-1 2.5) \rlvec (0 2.5)
       \move (6 2.5) \fcir f:0 r:0.3
       \move (6 2.5) \lvec (9 2.5)
       \htext (7.5 1){$2$}
       \move (6 2.5) \lvec (5 2)
       \move (6 2.5) \rlvec (2.5 1.25)
       \move (-2 2.5) \fcir f:0 r:0.3
       \move (-2 2.5) \lvec (-2.5 2.5)
       \htext (-2.25 1){$2$}
       \move (-2 2.5) \lvec (-2 2.4)
       \move (-2 2.5) \rlvec (0 2.5)
       \move (9 2.5) \fcir f:0 r:0.3
       \move (9 2.5) \lvec (8 2)
       \move (9 2.5) \rlvec (2.5 1.25)
       \move (9 2.5) \rlvec (2.5 0)
       \htext (10.5 2.5){$2$}
       \move (2 2) \fcir f:0 r:0.3
       \move (2 2) \lvec (-1 2)
       \htext (0.5 0.5){$3$}
       \move (2 2) \lvec (5 2)
       \htext (3.5 0.5){$2$}
       \move (2 2) \lvec (1.5 1.5)
       \move (-1 2) \fcir f:0 r:0.3
       \move (-1 2) \lvec (-2 2.4)
       \move (-1 2) \lvec (-0.75 1.62)
       \move (-2.5 2.5) \fcir f:0 r:0.3
       \move (-2.5 2.5) \lvec (-6 3)
       \move (-2.5 2.5) \lvec (-2 2.4)
       \move (-6 3) \fcir f:0 r:0.3
       \move (-6 3) \rlvec (-2.5 0)
       \move (-6 3) \rlvec (-2.5 0.41666666666666666666666666666666666666666666666667)
       \move (5 2) \fcir f:0 r:0.3
       \move (5 2) \lvec (8 2)
       \htext (6.5 0.5){$2$}
       \move (5 2) \lvec (4 1.5)
       \move (-2 2.4) \fcir f:0 r:0.3
       \move (8 2) \fcir f:0 r:0.3
       \move (8 2) \lvec (7 1.5)
       \move (8 2) \rlvec (2.5 0)
       \htext (9.5 2){$2$}
       \move (1.5 1.5) \fcir f:0 r:0.3
       \move (1.5 1.5) \lvec (4 1.5)
       \htext (2.75 0){$2$}
       \move (1.5 1.5) \lvec (-0.5 1.5)
       \htext (0.5 0){$2$}
       \move (1.5 1.5) \lvec (1 1)
       \move (4 1.5) \fcir f:0 r:0.3
       \move (4 1.5) \lvec (7 1.5)
       \htext (5.5 0){$2$}
       \move (4 1.5) \lvec (3 1)
       \move (-0.5 1.5) \fcir f:0 r:0.3
       \move (-0.5 1.5) \lvec (-0.75 1.62)
       \htext (-0.62 0.06){$2$}
       \move (-0.5 1.5) \lvec (0 1)
       \htext (-0.25 -0.25){$2$}
       \move (-0.75 1.62) \fcir f:0 r:0.3
       \move (-0.75 1.62) \lvec (-5 -0.5)
       \move (7 1.5) \fcir f:0 r:0.3
       \move (7 1.5) \lvec (6 1)
       \move (7 1.5) \rlvec (2.5 0)
       \htext (8.5 1.5){$2$}
       \move (1 1) \fcir f:0 r:0.3
       \move (1 1) \lvec (3 1)
       \htext (2 -0.5){$3$}
       \move (1 1) \lvec (0 1)
       \htext (0.5 -0.5){$2$}
       \move (1 1) \lvec (0 0.5)
       \move (3 1) \fcir f:0 r:0.3
       \move (3 1) \lvec (6 1)
       \htext (4.5 -0.5){$2$}
       \move (3 1) \lvec (2.5 0.5)
       \move (0 1) \fcir f:0 r:0.3
       \move (0 1) \lvec (0 0.5)
       \htext (0 -0.75){$2$}
       \move (6 1) \fcir f:0 r:0.3
       \move (6 1) \lvec (5 0.5)
       \move (6 1) \rlvec (2.5 0)
       \htext (7.5 1){$2$}
       \move (0 0.5) \fcir f:0 r:0.3
       \move (0 0.5) \lvec (0 0)
       \htext (0 -1.25){$2$}
       \move (0 0.5) \lvec (-1 0)
       \move (2.5 0.5) \fcir f:0 r:0.3
       \move (2.5 0.5) \lvec (5 0.5)
       \htext (3.75 -1){$2$}
       \move (2.5 0.5) \lvec (1 0)
       \move (5 0.5) \fcir f:0 r:0.3
       \move (5 0.5) \lvec (0.75 -1.62)
       \move (5 0.5) \rlvec (2.5 0)
       \htext (6.5 0.5){$2$}
       \move (1 0) \fcir f:0 r:0.3
       \move (1 0) \lvec (0 0)
       \move (1 0) \lvec (0 -0.5)
       \move (0 0) \fcir f:0 r:0.3
       \move (0 0) \lvec (0 -0.5)
       \htext (0 -1.75){$2$}
       \move (0 0) \lvec (-1 0)
       \move (-5 -0.5) \fcir f:0 r:0.3
       \move (-5 -0.5) \lvec (-2.5 -0.5)
       \htext (-3.75 -2){$2$}
       \move (-5 -0.5) \lvec (-6 -1)
       \move (-5 -0.5) \rlvec (-2.5 0)
       \htext (-6.5 -0.5){$2$}
       \move (0 -0.5) \fcir f:0 r:0.3
       \move (0 -0.5) \lvec (0 -1)
       \htext (0 -2.25){$2$}
       \move (0 -0.5) \lvec (-1 -1)
       \move (-1 0) \fcir f:0 r:0.3
       \move (-1 0) \lvec (-2.5 -0.5)
       \move (-2.5 -0.5) \fcir f:0 r:0.3
       \move (-2.5 -0.5) \lvec (-3 -1)
       \move (0.75 -1.62) \fcir f:0 r:0.3
       \move (0.75 -1.62) \lvec (0.5 -1.5)
       \htext (0.62 -3.06){$2$}
       \move (0.75 -1.62) \lvec (1 -2)
       \move (0 -1) \fcir f:0 r:0.3
       \move (0 -1) \lvec (0.5 -1.5)
       \htext (0.25 -2.75){$2$}
       \move (0 -1) \lvec (-1 -1)
       \htext (-0.5 -2.5){$2$}
       \move (0.5 -1.5) \fcir f:0 r:0.3
       \move (0.5 -1.5) \lvec (-1.5 -1.5)
       \htext (-0.5 -3){$2$}
       \move (-6 -1) \fcir f:0 r:0.3
       \move (-6 -1) \lvec (-3 -1)
       \htext (-4.5 -2.5){$2$}
       \move (-6 -1) \lvec (-7 -1.5)
       \move (-6 -1) \rlvec (-2.5 0)
       \htext (-7.5 -1){$2$}
       \move (6 -3) \fcir f:0 r:0.3
       \move (6 -3) \lvec (2.5 -2.5)
       \move (6 -3) \rlvec (2.5 0)
       \move (6 -3) \rlvec (2.5 -0.41666666666666666666666666666666666666666666666667)
       \move (1 -2) \fcir f:0 r:0.3
       \move (1 -2) \lvec (2 -2.4)
       \move (1 -2) \lvec (-2 -2)
       \htext (-0.5 -3.5){$3$}
       \move (1 -2) \lvec (1 -2.5)
       \move (2 -2.4) \fcir f:0 r:0.3
       \move (2 -2.4) \lvec (2.5 -2.5)
       \move (2 -2.4) \lvec (2 -2.5)
       \move (2.5 -2.5) \fcir f:0 r:0.3
       \move (2.5 -2.5) \lvec (2 -2.5)
       \htext (2.25 -4){$2$}
       \move (-3 -1) \fcir f:0 r:0.3
       \move (-3 -1) \lvec (-1 -1)
       \htext (-2 -2.5){$3$}
       \move (-3 -1) \lvec (-4 -1.5)
       \move (-1 -1) \fcir f:0 r:0.3
       \move (-1 -1) \lvec (-1.5 -1.5)
       \move (-7 -1.5) \fcir f:0 r:0.3
       \move (-7 -1.5) \lvec (-4 -1.5)
       \htext (-5.5 -3){$2$}
       \move (-7 -1.5) \lvec (-8 -2)
       \move (-7 -1.5) \rlvec (-2.5 0)
       \htext (-8.5 -1.5){$2$}
       \move (-4 -1.5) \fcir f:0 r:0.3
       \move (-4 -1.5) \lvec (-1.5 -1.5)
       \htext (-2.75 -3){$2$}
       \move (-4 -1.5) \lvec (-5 -2)
       \move (-1.5 -1.5) \fcir f:0 r:0.3
       \move (-1.5 -1.5) \lvec (-2 -2)
       \move (-8 -2) \fcir f:0 r:0.3
       \move (-8 -2) \lvec (-5 -2)
       \htext (-6.5 -3.5){$2$}
       \move (-8 -2) \lvec (-9 -2.5)
       \move (-8 -2) \rlvec (-2.5 0)
       \htext (-9.5 -2){$2$}
       \move (-5 -2) \fcir f:0 r:0.3
       \move (-5 -2) \lvec (-2 -2)
       \htext (-3.5 -3.5){$2$}
       \move (-5 -2) \lvec (-6 -2.5)
       \move (-2 -2) \fcir f:0 r:0.3
       \move (-2 -2) \lvec (-3 -2.5)
       \move (-9 -2.5) \fcir f:0 r:0.3
       \move (-9 -2.5) \lvec (-6 -2.5)
       \htext (-7.5 -4){$2$}
       \move (-9 -2.5) \rlvec (-2.5 -1.25)
       \move (-9 -2.5) \rlvec (-2.5 0)
       \htext (-10.5 -2.5){$2$}
       \move (2 -2.5) \fcir f:0 r:0.3
       \move (2 -2.5) \lvec (1 -2.5)
       \htext (1.5 -4){$2$}
       \move (2 -2.5) \rlvec (0 -2.5)
       \move (-6 -2.5) \fcir f:0 r:0.3
       \move (-6 -2.5) \lvec (-3 -2.5)
       \htext (-4.5 -4){$2$}
       \move (-6 -2.5) \rlvec (-2.5 -1.25)
       \move (1 -2.5) \fcir f:0 r:0.3
       \move (1 -2.5) \lvec (-3 -2.5)
       \htext (-1 -4){$2$}
       \move (1 -2.5) \rlvec (0 -2.5)
       \move (-3 -2.5) \fcir f:0 r:0.3
       \move (-3 -2.5) \rlvec (-2.5 -1.25)

   %% HERE STARTS THE CODE FOR THE LATTICE
        \move (-10 -4) \fcir f:0.8 r:0.15
        \move (-10 -3) \fcir f:0.8 r:0.15
        \move (-10 -2) \fcir f:0.8 r:0.15
        \move (-10 -1) \fcir f:0.8 r:0.15
        \move (-10 0) \fcir f:0.8 r:0.15
        \move (-10 1) \fcir f:0.8 r:0.15
        \move (-10 2) \fcir f:0.8 r:0.15
        \move (-10 3) \fcir f:0.8 r:0.15
        \move (-10 4) \fcir f:0.8 r:0.15
        \move (-9 -4) \fcir f:0.8 r:0.15
        \move (-9 -3) \fcir f:0.8 r:0.15
        \move (-9 -2) \fcir f:0.8 r:0.15
        \move (-9 -1) \fcir f:0.8 r:0.15
        \move (-9 0) \fcir f:0.8 r:0.15
        \move (-9 1) \fcir f:0.8 r:0.15
        \move (-9 2) \fcir f:0.8 r:0.15
        \move (-9 3) \fcir f:0.8 r:0.15
        \move (-9 4) \fcir f:0.8 r:0.15
        \move (-8 -4) \fcir f:0.8 r:0.15
        \move (-8 -3) \fcir f:0.8 r:0.15
        \move (-8 -2) \fcir f:0.8 r:0.15
        \move (-8 -1) \fcir f:0.8 r:0.15
        \move (-8 0) \fcir f:0.8 r:0.15
        \move (-8 1) \fcir f:0.8 r:0.15
        \move (-8 2) \fcir f:0.8 r:0.15
        \move (-8 3) \fcir f:0.8 r:0.15
        \move (-8 4) \fcir f:0.8 r:0.15
        \move (-7 -4) \fcir f:0.8 r:0.15
        \move (-7 -3) \fcir f:0.8 r:0.15
        \move (-7 -2) \fcir f:0.8 r:0.15
        \move (-7 -1) \fcir f:0.8 r:0.15
        \move (-7 0) \fcir f:0.8 r:0.15
        \move (-7 1) \fcir f:0.8 r:0.15
        \move (-7 2) \fcir f:0.8 r:0.15
        \move (-7 3) \fcir f:0.8 r:0.15
        \move (-7 4) \fcir f:0.8 r:0.15
        \move (-6 -4) \fcir f:0.8 r:0.15
        \move (-6 -3) \fcir f:0.8 r:0.15
        \move (-6 -2) \fcir f:0.8 r:0.15
        \move (-6 -1) \fcir f:0.8 r:0.15
        \move (-6 0) \fcir f:0.8 r:0.15
        \move (-6 1) \fcir f:0.8 r:0.15
        \move (-6 2) \fcir f:0.8 r:0.15
        \move (-6 3) \fcir f:0.8 r:0.15
        \move (-6 4) \fcir f:0.8 r:0.15
        \move (-5 -4) \fcir f:0.8 r:0.15
        \move (-5 -3) \fcir f:0.8 r:0.15
        \move (-5 -2) \fcir f:0.8 r:0.15
        \move (-5 -1) \fcir f:0.8 r:0.15
        \move (-5 0) \fcir f:0.8 r:0.15
        \move (-5 1) \fcir f:0.8 r:0.15
        \move (-5 2) \fcir f:0.8 r:0.15
        \move (-5 3) \fcir f:0.8 r:0.15
        \move (-5 4) \fcir f:0.8 r:0.15
        \move (-4 -4) \fcir f:0.8 r:0.15
        \move (-4 -3) \fcir f:0.8 r:0.15
        \move (-4 -2) \fcir f:0.8 r:0.15
        \move (-4 -1) \fcir f:0.8 r:0.15
        \move (-4 0) \fcir f:0.8 r:0.15
        \move (-4 1) \fcir f:0.8 r:0.15
        \move (-4 2) \fcir f:0.8 r:0.15
        \move (-4 3) \fcir f:0.8 r:0.15
        \move (-4 4) \fcir f:0.8 r:0.15
        \move (-3 -4) \fcir f:0.8 r:0.15
        \move (-3 -3) \fcir f:0.8 r:0.15
        \move (-3 -2) \fcir f:0.8 r:0.15
        \move (-3 -1) \fcir f:0.8 r:0.15
        \move (-3 0) \fcir f:0.8 r:0.15
        \move (-3 1) \fcir f:0.8 r:0.15
        \move (-3 2) \fcir f:0.8 r:0.15
        \move (-3 3) \fcir f:0.8 r:0.15
        \move (-3 4) \fcir f:0.8 r:0.15
        \move (-2 -4) \fcir f:0.8 r:0.15
        \move (-2 -3) \fcir f:0.8 r:0.15
        \move (-2 -2) \fcir f:0.8 r:0.15
        \move (-2 -1) \fcir f:0.8 r:0.15
        \move (-2 0) \fcir f:0.8 r:0.15
        \move (-2 1) \fcir f:0.8 r:0.15
        \move (-2 2) \fcir f:0.8 r:0.15
        \move (-2 3) \fcir f:0.8 r:0.15
        \move (-2 4) \fcir f:0.8 r:0.15
        \move (-1 -4) \fcir f:0.8 r:0.15
        \move (-1 -3) \fcir f:0.8 r:0.15
        \move (-1 -2) \fcir f:0.8 r:0.15
        \move (-1 -1) \fcir f:0.8 r:0.15
        \move (-1 0) \fcir f:0.8 r:0.15
        \move (-1 1) \fcir f:0.8 r:0.15
        \move (-1 2) \fcir f:0.8 r:0.15
        \move (-1 3) \fcir f:0.8 r:0.15
        \move (-1 4) \fcir f:0.8 r:0.15
        \move (0 -4) \fcir f:0.8 r:0.15
        \move (0 -3) \fcir f:0.8 r:0.15
        \move (0 -2) \fcir f:0.8 r:0.15
        \move (0 -1) \fcir f:0.8 r:0.15
        \move (0 0) \fcir f:0.8 r:0.15
        \move (0 1) \fcir f:0.8 r:0.15
        \move (0 2) \fcir f:0.8 r:0.15
        \move (0 3) \fcir f:0.8 r:0.15
        \move (0 4) \fcir f:0.8 r:0.15
        \move (1 -4) \fcir f:0.8 r:0.15
        \move (1 -3) \fcir f:0.8 r:0.15
        \move (1 -2) \fcir f:0.8 r:0.15
        \move (1 -1) \fcir f:0.8 r:0.15
        \move (1 0) \fcir f:0.8 r:0.15
        \move (1 1) \fcir f:0.8 r:0.15
        \move (1 2) \fcir f:0.8 r:0.15
        \move (1 3) \fcir f:0.8 r:0.15
        \move (1 4) \fcir f:0.8 r:0.15
        \move (2 -4) \fcir f:0.8 r:0.15
        \move (2 -3) \fcir f:0.8 r:0.15
        \move (2 -2) \fcir f:0.8 r:0.15
        \move (2 -1) \fcir f:0.8 r:0.15
        \move (2 0) \fcir f:0.8 r:0.15
        \move (2 1) \fcir f:0.8 r:0.15
        \move (2 2) \fcir f:0.8 r:0.15
        \move (2 3) \fcir f:0.8 r:0.15
        \move (2 4) \fcir f:0.8 r:0.15
        \move (3 -4) \fcir f:0.8 r:0.15
        \move (3 -3) \fcir f:0.8 r:0.15
        \move (3 -2) \fcir f:0.8 r:0.15
        \move (3 -1) \fcir f:0.8 r:0.15
        \move (3 0) \fcir f:0.8 r:0.15
        \move (3 1) \fcir f:0.8 r:0.15
        \move (3 2) \fcir f:0.8 r:0.15
        \move (3 3) \fcir f:0.8 r:0.15
        \move (3 4) \fcir f:0.8 r:0.15
        \move (4 -4) \fcir f:0.8 r:0.15
        \move (4 -3) \fcir f:0.8 r:0.15
        \move (4 -2) \fcir f:0.8 r:0.15
        \move (4 -1) \fcir f:0.8 r:0.15
        \move (4 0) \fcir f:0.8 r:0.15
        \move (4 1) \fcir f:0.8 r:0.15
        \move (4 2) \fcir f:0.8 r:0.15
        \move (4 3) \fcir f:0.8 r:0.15
        \move (4 4) \fcir f:0.8 r:0.15
        \move (5 -4) \fcir f:0.8 r:0.15
        \move (5 -3) \fcir f:0.8 r:0.15
        \move (5 -2) \fcir f:0.8 r:0.15
        \move (5 -1) \fcir f:0.8 r:0.15
        \move (5 0) \fcir f:0.8 r:0.15
        \move (5 1) \fcir f:0.8 r:0.15
        \move (5 2) \fcir f:0.8 r:0.15
        \move (5 3) \fcir f:0.8 r:0.15
        \move (5 4) \fcir f:0.8 r:0.15
        \move (6 -4) \fcir f:0.8 r:0.15
        \move (6 -3) \fcir f:0.8 r:0.15
        \move (6 -2) \fcir f:0.8 r:0.15
        \move (6 -1) \fcir f:0.8 r:0.15
        \move (6 0) \fcir f:0.8 r:0.15
        \move (6 1) \fcir f:0.8 r:0.15
        \move (6 2) \fcir f:0.8 r:0.15
        \move (6 3) \fcir f:0.8 r:0.15
        \move (6 4) \fcir f:0.8 r:0.15
        \move (7 -4) \fcir f:0.8 r:0.15
        \move (7 -3) \fcir f:0.8 r:0.15
        \move (7 -2) \fcir f:0.8 r:0.15
        \move (7 -1) \fcir f:0.8 r:0.15
        \move (7 0) \fcir f:0.8 r:0.15
        \move (7 1) \fcir f:0.8 r:0.15
        \move (7 2) \fcir f:0.8 r:0.15
        \move (7 3) \fcir f:0.8 r:0.15
        \move (7 4) \fcir f:0.8 r:0.15
        \move (8 -4) \fcir f:0.8 r:0.15
        \move (8 -3) \fcir f:0.8 r:0.15
        \move (8 -2) \fcir f:0.8 r:0.15
        \move (8 -1) \fcir f:0.8 r:0.15
        \move (8 0) \fcir f:0.8 r:0.15
        \move (8 1) \fcir f:0.8 r:0.15
        \move (8 2) \fcir f:0.8 r:0.15
        \move (8 3) \fcir f:0.8 r:0.15
        \move (8 4) \fcir f:0.8 r:0.15
        \move (9 -4) \fcir f:0.8 r:0.15
        \move (9 -3) \fcir f:0.8 r:0.15
        \move (9 -2) \fcir f:0.8 r:0.15
        \move (9 -1) \fcir f:0.8 r:0.15
        \move (9 0) \fcir f:0.8 r:0.15
        \move (9 1) \fcir f:0.8 r:0.15
        \move (9 2) \fcir f:0.8 r:0.15
        \move (9 3) \fcir f:0.8 r:0.15
        \move (9 4) \fcir f:0.8 r:0.15
        \move (10 -4) \fcir f:0.8 r:0.15
        \move (10 -3) \fcir f:0.8 r:0.15
        \move (10 -2) \fcir f:0.8 r:0.15
        \move (10 -1) \fcir f:0.8 r:0.15
        \move (10 0) \fcir f:0.8 r:0.15
        \move (10 1) \fcir f:0.8 r:0.15
        \move (10 2) \fcir f:0.8 r:0.15
        \move (10 3) \fcir f:0.8 r:0.15
        \move (10 4) \fcir f:0.8 r:0.15
   %% HERE ENDS THE CODE FOR THE LATTICE
                          
    \end{texdraw}\end{center}

The Newton subdivision of the tropical curve is:

   
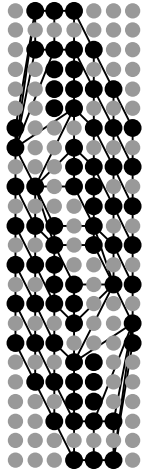
\vspace*{0.5cm}

   \begin{center}
            
    \begin{texdraw}
       \drawdim cm  \relunitscale 0.5  %% originally: 1 
%      \drawdim cm  \relunitscale 1 
       \linewd 0.05
       \relunitscale0.52
        \move (3 23)        
        \lvec (6 17)
        \move (6 17)        
        \lvec (6 6)
        \move (6 6)        
        \lvec (5 0)
        \move (5 0)        
        \lvec (3 0)
        \move (3 0)        
        \lvec (0 6)
        \move (0 6)        
        \lvec (0 17)
        \move (0 17)        
        \lvec (1 23)
        \move (1 23)        
        \lvec (3 23)

        \move (3 21)        
        \lvec (3 23)
        \move (4 21)        
        \lvec (4 19)
        \move (4 19)        
        \lvec (3 21)
        \move (2 21)        
        \lvec (2 23)
        \move (3 21)        
        \lvec (2 21)
        \move (5 19)        
        \lvec (5 17)
        \move (5 17)        
        \lvec (4 19)
        \move (1 21)        
        \lvec (1 23)
        \move (2 21)        
        \lvec (1 21)
        \move (6 15)        
        \lvec (5 17)
        \move (3 20)        
        \lvec (3 21)
        \move (3 19)        
        \lvec (3 20)
        \move (3 18)        
        \lvec (3 19)
        \move (4 19)        
        \lvec (4 17)
        \move (4 17)        
        \lvec (3 18)
        \move (0 16)        
        \lvec (2 21)
        \move (3 18)        
        \lvec (0 16)
        \move (0 16)        
        \lvec (1 23)
        \move (1 21)        
        \lvec (0 16)
        \move (5 17)        
        \lvec (5 15)
        \move (5 15)        
        \lvec (4 17)
        \move (6 13)        
        \lvec (5 15)
        \move (4 17)        
        \lvec (4 15)
        \move (3 16)        
        \lvec (3 18)
        \move (4 15)        
        \lvec (3 16)
        \move (5 15)        
        \lvec (5 13)
        \move (5 13)        
        \lvec (4 15)
        \move (1 14)        
        \lvec (3 18)
        \move (3 16)        
        \lvec (1 14)
        \move (1 14)        
        \lvec (0 16)
        \move (6 11)        
        \lvec (5 13)
        \move (4 13)        
        \lvec (4 12)
        \move (4 14)        
        \lvec (4 13)
        \move (4 15)        
        \lvec (4 14)
        \move (3 15)        
        \lvec (3 16)
        \move (3 14)        
        \lvec (3 15)
        \move (4 12)        
        \lvec (3 14)
        \move (5 13)        
        \lvec (5 11)
        \move (5 11)        
        \lvec (4 12)
        \move (3 14)        
        \lvec (1 14)
        \move (6 9)        
        \lvec (5 11)
        \move (4 12)        
        \lvec (2 12)
        \move (2 12)        
        \lvec (1 14)
        \move (5 11)        
        \lvec (5 9)
        \move (5 9)        
        \lvec (4 12)
        \move (6 7)        
        \lvec (5 9)
        \move (4 11)        
        \lvec (4 12)
        \move (5 9)        
        \lvec (4 11)
        \move (4 11)        
        \lvec (2 11)
        \move (2 11)        
        \lvec (2 12)
        \move (1 14)        
        \lvec (1 12)
        \move (1 12)        
        \lvec (0 14)
        \move (5 9)        
        \lvec (3 9)
        \move (3 9)        
        \lvec (2 11)
        \move (2 11)        
        \lvec (1 14)
        \move (2 11)        
        \lvec (1 12)
        \move (3 5)        
        \lvec (5 9)
        \move (6 7)        
        \lvec (3 5)
        \move (5 9)        
        \lvec (3 7)
        \move (3 8)        
        \lvec (3 9)
        \move (3 7)        
        \lvec (3 8)
        \move (3 5)        
        \lvec (3 7)
        \move (1 12)        
        \lvec (1 10)
        \move (1 10)        
        \lvec (0 12)
        \move (5 0)        
        \lvec (6 7)
        \move (6 7)        
        \lvec (4 2)
        \move (3 4)        
        \lvec (3 5)
        \move (3 3)        
        \lvec (3 4)
        \move (3 2)        
        \lvec (3 3)
        \move (4 2)        
        \lvec (3 2)
        \move (6 7)        
        \lvec (5 2)
        \move (5 2)        
        \lvec (4 2)
        \move (5 0)        
        \lvec (5 2)
        \move (2 9)        
        \lvec (2 8)
        \move (2 10)        
        \lvec (2 9)
        \move (2 11)        
        \lvec (2 10)
        \move (2 8)        
        \lvec (1 10)
        \move (3 7)        
        \lvec (2 8)
        \move (1 10)        
        \lvec (1 8)
        \move (1 8)        
        \lvec (0 10)
        \move (2 8)        
        \lvec (2 6)
        \move (2 6)        
        \lvec (1 8)
        \move (3 5)        
        \lvec (2 6)
        \move (1 8)        
        \lvec (1 6)
        \move (1 6)        
        \lvec (0 8)
        \move (2 6)        
        \lvec (2 4)
        \move (2 4)        
        \lvec (1 6)
        \move (3 2)        
        \lvec (2 4)
        \move (1 6)        
        \lvec (1 4)
        \move (4 0)        
        \lvec (4 2)
        \move (2 4)        
        \lvec (2 2)
        \move (3 0)        
        \lvec (3 2)
        \move (0 0) \fcir f:0.6 r:0.38
        \move (0 1) \fcir f:0.6 r:0.38
        \move (0 2) \fcir f:0.6 r:0.38
        \move (0 3) \fcir f:0.6 r:0.38
        \move (0 4) \fcir f:0.6 r:0.38
        \move (0 5) \fcir f:0.6 r:0.38
        \move (0 6) \fcir f:0.6 r:0.38
        \move (0 7) \fcir f:0.6 r:0.38
        \move (0 8) \fcir f:0.6 r:0.38
        \move (0 9) \fcir f:0.6 r:0.38
        \move (0 10) \fcir f:0.6 r:0.38
        \move (0 11) \fcir f:0.6 r:0.38
        \move (0 12) \fcir f:0.6 r:0.38
        \move (0 13) \fcir f:0.6 r:0.38
        \move (0 14) \fcir f:0.6 r:0.38
        \move (0 15) \fcir f:0.6 r:0.38
        \move (0 16) \fcir f:0.6 r:0.38
        \move (0 17) \fcir f:0.6 r:0.38
        \move (0 18) \fcir f:0.6 r:0.38
        \move (0 19) \fcir f:0.6 r:0.38
        \move (0 20) \fcir f:0.6 r:0.38
        \move (0 21) \fcir f:0.6 r:0.38
        \move (0 22) \fcir f:0.6 r:0.38
        \move (0 23) \fcir f:0.6 r:0.38
        \move (1 0) \fcir f:0.6 r:0.38
        \move (1 1) \fcir f:0.6 r:0.38
        \move (1 2) \fcir f:0.6 r:0.38
        \move (1 3) \fcir f:0.6 r:0.38
        \move (1 4) \fcir f:0.6 r:0.38
        \move (1 5) \fcir f:0.6 r:0.38
        \move (1 6) \fcir f:0.6 r:0.38
        \move (1 7) \fcir f:0.6 r:0.38
        \move (1 8) \fcir f:0.6 r:0.38
        \move (1 9) \fcir f:0.6 r:0.38
        \move (1 10) \fcir f:0.6 r:0.38
        \move (1 11) \fcir f:0.6 r:0.38
        \move (1 12) \fcir f:0.6 r:0.38
        \move (1 13) \fcir f:0.6 r:0.38
        \move (1 14) \fcir f:0.6 r:0.38
        \move (1 15) \fcir f:0.6 r:0.38
        \move (1 16) \fcir f:0.6 r:0.38
        \move (1 17) \fcir f:0.6 r:0.38
        \move (1 18) \fcir f:0.6 r:0.38
        \move (1 19) \fcir f:0.6 r:0.38
        \move (1 20) \fcir f:0.6 r:0.38
        \move (1 21) \fcir f:0.6 r:0.38
        \move (1 22) \fcir f:0.6 r:0.38
        \move (1 23) \fcir f:0.6 r:0.38
        \move (2 0) \fcir f:0.6 r:0.38
        \move (2 1) \fcir f:0.6 r:0.38
        \move (2 2) \fcir f:0.6 r:0.38
        \move (2 3) \fcir f:0.6 r:0.38
        \move (2 4) \fcir f:0.6 r:0.38
        \move (2 5) \fcir f:0.6 r:0.38
        \move (2 6) \fcir f:0.6 r:0.38
        \move (2 7) \fcir f:0.6 r:0.38
        \move (2 8) \fcir f:0.6 r:0.38
        \move (2 9) \fcir f:0.6 r:0.38
        \move (2 10) \fcir f:0.6 r:0.38
        \move (2 11) \fcir f:0.6 r:0.38
        \move (2 12) \fcir f:0.6 r:0.38
        \move (2 13) \fcir f:0.6 r:0.38
        \move (2 14) \fcir f:0.6 r:0.38
        \move (2 15) \fcir f:0.6 r:0.38
        \move (2 16) \fcir f:0.6 r:0.38
        \move (2 17) \fcir f:0.6 r:0.38
        \move (2 18) \fcir f:0.6 r:0.38
        \move (2 19) \fcir f:0.6 r:0.38
        \move (2 20) \fcir f:0.6 r:0.38
        \move (2 21) \fcir f:0.6 r:0.38
        \move (2 22) \fcir f:0.6 r:0.38
        \move (2 23) \fcir f:0.6 r:0.38
        \move (3 0) \fcir f:0.6 r:0.38
        \move (3 1) \fcir f:0.6 r:0.38
        \move (3 2) \fcir f:0.6 r:0.38
        \move (3 3) \fcir f:0.6 r:0.38
        \move (3 4) \fcir f:0.6 r:0.38
        \move (3 5) \fcir f:0.6 r:0.38
        \move (3 6) \fcir f:0.6 r:0.38
        \move (3 7) \fcir f:0.6 r:0.38
        \move (3 8) \fcir f:0.6 r:0.38
        \move (3 9) \fcir f:0.6 r:0.38
        \move (3 10) \fcir f:0.6 r:0.38
        \move (3 11) \fcir f:0.6 r:0.38
        \move (3 12) \fcir f:0.6 r:0.38
        \move (3 13) \fcir f:0.6 r:0.38
        \move (3 14) \fcir f:0.6 r:0.38
        \move (3 15) \fcir f:0.6 r:0.38
        \move (3 16) \fcir f:0.6 r:0.38
        \move (3 17) \fcir f:0.6 r:0.38
        \move (3 18) \fcir f:0.6 r:0.38
        \move (3 19) \fcir f:0.6 r:0.38
        \move (3 20) \fcir f:0.6 r:0.38
        \move (3 21) \fcir f:0.6 r:0.38
        \move (3 22) \fcir f:0.6 r:0.38
        \move (3 23) \fcir f:0.6 r:0.38
        \move (4 0) \fcir f:0.6 r:0.38
        \move (4 1) \fcir f:0.6 r:0.38
        \move (4 2) \fcir f:0.6 r:0.38
        \move (4 3) \fcir f:0.6 r:0.38
        \move (4 4) \fcir f:0.6 r:0.38
        \move (4 5) \fcir f:0.6 r:0.38
        \move (4 6) \fcir f:0.6 r:0.38
        \move (4 7) \fcir f:0.6 r:0.38
        \move (4 8) \fcir f:0.6 r:0.38
        \move (4 9) \fcir f:0.6 r:0.38
        \move (4 10) \fcir f:0.6 r:0.38
        \move (4 11) \fcir f:0.6 r:0.38
        \move (4 12) \fcir f:0.6 r:0.38
        \move (4 13) \fcir f:0.6 r:0.38
        \move (4 14) \fcir f:0.6 r:0.38
        \move (4 15) \fcir f:0.6 r:0.38
        \move (4 16) \fcir f:0.6 r:0.38
        \move (4 17) \fcir f:0.6 r:0.38
        \move (4 18) \fcir f:0.6 r:0.38
        \move (4 19) \fcir f:0.6 r:0.38
        \move (4 20) \fcir f:0.6 r:0.38
        \move (4 21) \fcir f:0.6 r:0.38
        \move (4 22) \fcir f:0.6 r:0.38
        \move (4 23) \fcir f:0.6 r:0.38
        \move (5 0) \fcir f:0.6 r:0.38
        \move (5 1) \fcir f:0.6 r:0.38
        \move (5 2) \fcir f:0.6 r:0.38
        \move (5 3) \fcir f:0.6 r:0.38
        \move (5 4) \fcir f:0.6 r:0.38
        \move (5 5) \fcir f:0.6 r:0.38
        \move (5 6) \fcir f:0.6 r:0.38
        \move (5 7) \fcir f:0.6 r:0.38
        \move (5 8) \fcir f:0.6 r:0.38
        \move (5 9) \fcir f:0.6 r:0.38
        \move (5 10) \fcir f:0.6 r:0.38
        \move (5 11) \fcir f:0.6 r:0.38
        \move (5 12) \fcir f:0.6 r:0.38
        \move (5 13) \fcir f:0.6 r:0.38
        \move (5 14) \fcir f:0.6 r:0.38
        \move (5 15) \fcir f:0.6 r:0.38
        \move (5 16) \fcir f:0.6 r:0.38
        \move (5 17) \fcir f:0.6 r:0.38
        \move (5 18) \fcir f:0.6 r:0.38
        \move (5 19) \fcir f:0.6 r:0.38
        \move (5 20) \fcir f:0.6 r:0.38
        \move (5 21) \fcir f:0.6 r:0.38
        \move (5 22) \fcir f:0.6 r:0.38
        \move (5 23) \fcir f:0.6 r:0.38
        \move (6 0) \fcir f:0.6 r:0.38
        \move (6 1) \fcir f:0.6 r:0.38
        \move (6 2) \fcir f:0.6 r:0.38
        \move (6 3) \fcir f:0.6 r:0.38
        \move (6 4) \fcir f:0.6 r:0.38
        \move (6 5) \fcir f:0.6 r:0.38
        \move (6 6) \fcir f:0.6 r:0.38
        \move (6 7) \fcir f:0.6 r:0.38
        \move (6 8) \fcir f:0.6 r:0.38
        \move (6 9) \fcir f:0.6 r:0.38
        \move (6 10) \fcir f:0.6 r:0.38
        \move (6 11) \fcir f:0.6 r:0.38
        \move (6 12) \fcir f:0.6 r:0.38
        \move (6 13) \fcir f:0.6 r:0.38
        \move (6 14) \fcir f:0.6 r:0.38
        \move (6 15) \fcir f:0.6 r:0.38
        \move (6 16) \fcir f:0.6 r:0.38
        \move (6 17) \fcir f:0.6 r:0.38
        \move (6 18) \fcir f:0.6 r:0.38
        \move (6 19) \fcir f:0.6 r:0.38
        \move (6 20) \fcir f:0.6 r:0.38
        \move (6 21) \fcir f:0.6 r:0.38
        \move (6 22) \fcir f:0.6 r:0.38
        \move (6 23) \fcir f:0.6 r:0.38
       \move (3 23) 
       \fcir f:0 r:0.47
       \move (4 21) 
       \fcir f:0 r:0.47
       \move (3 21) 
       \fcir f:0 r:0.47
       \move (4 19) 
       \fcir f:0 r:0.47
       \move (2 23) 
       \fcir f:0 r:0.47
       \move (2 21) 
       \fcir f:0 r:0.47
       \move (5 19) 
       \fcir f:0 r:0.47
       \move (5 17) 
       \fcir f:0 r:0.47
       \move (1 23) 
       \fcir f:0 r:0.47
       \move (1 21) 
       \fcir f:0 r:0.47
       \move (6 17) 
       \fcir f:0 r:0.47
       \move (6 15) 
       \fcir f:0 r:0.47
       \move (3 20) 
       \fcir f:0 r:0.47
       \move (3 19) 
       \fcir f:0 r:0.47
       \move (4 17) 
       \fcir f:0 r:0.47
       \move (3 18) 
       \fcir f:0 r:0.47
       \move (2 20) 
       \fcir f:0 r:0.47
       \move (2 19) 
       \fcir f:0 r:0.47
       \move (2 18) 
       \fcir f:0 r:0.47
       \move (0 16) 
       \fcir f:0 r:0.47
       \move (0 17) 
       \fcir f:0 r:0.47
       \move (5 15) 
       \fcir f:0 r:0.47
       \move (6 13) 
       \fcir f:0 r:0.47
       \move (4 15) 
       \fcir f:0 r:0.47
       \move (3 16) 
       \fcir f:0 r:0.47
       \move (5 13) 
       \fcir f:0 r:0.47
       \move (1 14) 
       \fcir f:0 r:0.47
       \move (6 11) 
       \fcir f:0 r:0.47
       \move (4 14) 
       \fcir f:0 r:0.47
       \move (3 15) 
       \fcir f:0 r:0.47
       \move (4 13) 
       \fcir f:0 r:0.47
       \move (3 14) 
       \fcir f:0 r:0.47
       \move (4 12) 
       \fcir f:0 r:0.47
       \move (5 11) 
       \fcir f:0 r:0.47
       \move (6 9) 
       \fcir f:0 r:0.47
       \move (2 12) 
       \fcir f:0 r:0.47
       \move (5 9) 
       \fcir f:0 r:0.47
       \move (6 7) 
       \fcir f:0 r:0.47
       \move (4 11) 
       \fcir f:0 r:0.47
       \move (2 11) 
       \fcir f:0 r:0.47
       \move (0 14) 
       \fcir f:0 r:0.47
       \move (1 12) 
       \fcir f:0 r:0.47
       \move (3 9) 
       \fcir f:0 r:0.47
       \move (3 5) 
       \fcir f:0 r:0.47
       \move (3 8) 
       \fcir f:0 r:0.47
       \move (3 7) 
       \fcir f:0 r:0.47
       \move (0 12) 
       \fcir f:0 r:0.47
       \move (1 10) 
       \fcir f:0 r:0.47
       \move (6 6) 
       \fcir f:0 r:0.47
       \move (5 0) 
       \fcir f:0 r:0.47
       \move (4 5) 
       \fcir f:0 r:0.47
       \move (4 4) 
       \fcir f:0 r:0.47
       \move (4 3) 
       \fcir f:0 r:0.47
       \move (3 4) 
       \fcir f:0 r:0.47
       \move (4 2) 
       \fcir f:0 r:0.47
       \move (3 3) 
       \fcir f:0 r:0.47
       \move (3 2) 
       \fcir f:0 r:0.47
       \move (5 2) 
       \fcir f:0 r:0.47
       \move (2 10) 
       \fcir f:0 r:0.47
       \move (2 9) 
       \fcir f:0 r:0.47
       \move (2 8) 
       \fcir f:0 r:0.47
       \move (0 10) 
       \fcir f:0 r:0.47
       \move (1 8) 
       \fcir f:0 r:0.47
       \move (2 6) 
       \fcir f:0 r:0.47
       \move (0 8) 
       \fcir f:0 r:0.47
       \move (1 6) 
       \fcir f:0 r:0.47
       \move (2 4) 
       \fcir f:0 r:0.47
       \move (0 6) 
       \fcir f:0 r:0.47
       \move (1 4) 
       \fcir f:0 r:0.47
       \move (4 0) 
       \fcir f:0 r:0.47
       \move (2 2) 
       \fcir f:0 r:0.47
       \move (3 0) 
       \fcir f:0 r:0.47
   \end{texdraw}
     
   \end{center}

\subsection{The number of terms of the  non-homogeneous $A$-polynomial
of twist knots}
\lbl{sub.terms}

In \cite{GS} we explicitly computed the non-homogeneous $A$-polynomial 
$(A^{nh}_{K_p},B_{K_p})$ of the {\em twist knots} $K_p$ for $p=-8,\dots,11$.
$K_p$ is the knot obtained by $1/p$ surgery on one component of the 
Whitehead link. 
This includes the following knots in the Rolfsen notation:
$$
K_1 = 3_1, K_2 = 5_2, K_3 = 7_2, K_4 = 9_2,
\qquad
K_{-1} = 4_1, K_{-2} = 6_1, K_{-3} = 8_1, K_{-4} = 10_1.
$$ 
The computations reveal that for $p=1,\dots,11$, $A^{nh}_{K_p}$ has 
$(L,M,q)$ degree equal to 
$$
\left(2p-1,8p-4,\frac{17}{2} p(p-1) + 2\right)
$$ 
The total number of terms of the 3-variable polynomial $A^{nh}_{K_p}$
is given by
$$
139976, 80252, 41996, 19402, 7406, 2112, 346, 22
$$
for $p=-8,\dots,-1$, and by
$$
4, 98, 908, 4100, 12236, 28978, 58668, 106800, 179814, 284998, 430652
$$
for $p=1,\dots,11$.
%%% see mathematica file: AqPolyTwistKnots.nb 
Using the data from \cite{GS}, the author has computed the tropical curves 
(homogeneous or not) of all twist knots $K_p$ with $p=-8,\dots,11$.
Needless to say, the output of the 
computations it too large to be displayed in the paper.

\subsection{Acknowledgment}
The idea of the present paper was conceived during the New York
Conference on {\em Interactions between Hyperbolic Geometry, 
Quantum Topology and Number Theory} in New York in the summer of 2009.
An early version of the present paper appeared in the New Zealand Conference
on {\em Topological Quantum Field Theory and Knot Homology Theory}
in January 2010. 
The author wishes to thank the organizers of the New York Conference, 
A. Champanerkar, O. Dasbach, E. Kalfagianni, I. Kofman, W. Neumann and 
N. Stoltzfus and the New Zealand Conference  R. Fenn, D. Gauld
and V. Jones  for their hospitality and for creating a
stimulating atmosphere.
The author also wishes to thank J. Yu for many enlightening conversations and
T. Markwig for the drawing implementation of {\tt polymake}.

%{\bf
%
%\section{Todo}
%
%\begin{itemize}
%\item
%nothing
%\end{itemize}
%}

\ifx\undefined\bysame
        \newcommand{\bysame}{\leavevmode\hbox
to3em{\hrulefill}\,}
\fi

\end{document}